\documentclass[12pt,reqno]{amsart}
\usepackage{amscd,amsmath,amsthm,amssymb}
\usepackage{color}
\usepackage{tikz}
\usepackage{url}
\usepackage{circuitikz}
\usepackage{float}

\newtheorem{Theorem}{Theorem}[section]
\newtheorem{Lemma}[Theorem]{Lemma}
\newtheorem{Corollary}[Theorem]{Corollary}
\newtheorem{Proposition}[Theorem]{Proposition}
\newtheorem{Remark}[Theorem]{Remark}

\newtheorem{Question}[Theorem]{Question}

\newtheorem{Conj}[Theorem]{Conjecture}

\newtheorem{Conjecture}{Conjecture}

\def\qed{\ifhmode\textqed\fi
	\ifmmode\ifinner\quad\qedsymbol\else\dispqed\fi\fi}
\def\textqed{\unskip\nobreak\penalty50
	\hskip2em\hbox{}\nobreak\hfill\qedsymbol
	\parfillskip=0pt \finalhyphendemerits=0}
\def\dispqed{\rlap{\qquad\qedsymbol}}

\def\Ker{\textup{Ker}}
\def\ZZ{\mathbb{Z}}
\def\m{\mathfrak{m}}

\def\height{\textup{height}}
\def\Ass{\textup{Ass}}

\def\FF{\mathbb{F}}
\def\Tor{\textup{Tor}}
\def\depth{\textup{depth\,}}

\def\pd{\textup{proj\,dim}}
\def\supp{\textup{supp}}
\def\reg{\textup{reg}}
\def\v{\textup{v}}
\def\lex{\textup{lex}}
\def\rlex{\textup{rlex}}
\def\ini{\textup{in}}
\def\Im{\textup{Im}}
\def\Match{\textup{Match}}
\def\Spec{\textup{Spec}}
\def\bideg{\textup{bideg}}

\begin{document}
	
	\title{Symbolic powers of polymatroidal ideals}
	\author{Antonino Ficarra, Somayeh Moradi}
	
	\address{Antonino Ficarra, Departamento de Matem\'{a}tica, Escola de Ci\^{e}ncias e Tecnologia, Centro de Investiga\c{c}\~{a}o, Matem\'{a}tica e Aplica\c{c}\~{o}es, Instituto de Investiga\c{c}\~{a}o e Forma\c{c}\~{a}o Avan\c{c}ada, Universidade de \'{E}vora, Rua Rom\~{a}o Ramalho, 59, P--7000--671 \'{E}vora, Portugal}
	\email{antonino.ficarra@uevora.pt\,\,,\quad antficarra@unime.it}
	
	\address{Somayeh Moradi, Department of Mathematics, Faculty of Science, Ilam University, P.O.Box 69315-516, Ilam, Iran}
	\email{so.moradi@ilam.ac.ir}

	\subjclass[2020]{Primary 13D02, 13C05, 13A02; Secondary 05E40}
	\keywords{Symbolic power, componentwise linear, regularity, polymatroidal ideal}
	
	\begin{abstract}
		In this paper, we investigate the componentwise linearity and the Castelnuovo-Mumford regularity  of symbolic powers of polymatroidal ideals. For a polymatroidal ideal $I$, we conjecture that every symbolic power $I^{(k)}$ is componentwise linear and
		$$
		\text{reg}\,I^{(k)}=\text{reg}\,I^k
		$$
		for all $k \ge 1$. We prove that $\text{reg}\,I^{(k)}\ge\text{reg}\,I^k$ for all $k \ge 1$ when $I$ has no embedded associated primes, for instance if $I$ is a matroidal ideal. Moreover, we establish a criterion on the symbolic Rees algebra $\mathcal{R}_s(I)$ of a monomial ideal of minimal intersection type which guarantees that every symbolic power $I^{(k)}$ has linear quotients and, hence, is componentwise linear for all $k\ge1$. By applying our criterion to squarefree Veronese ideals and certain matching-matroidal ideals, we verify both conjectures for these families. We establish the Conforti-Cornu\'ejols conjecture for any matroidal ideal, and we show that a matroidal ideal is packed if and only if it is the product of monomial prime ideals with pairwise disjoint supports. Furthermore, we identify several classes of non-squarefree polymatroidal ideals for which the ordinary and symbolic powers coincide. Hence, we confirm our conjectures for transversal polymatroidal ideals and principal Borel ideals. Finally, we verify our conjectures for all polymatroidal ideals either generated in small degrees or in a small number of variables.
	\end{abstract}
	\dedicatory{Dedicated to the memory of Jürgen Herzog,\\ whose passion for mathematics continues to inspire}
	
	\maketitle
	\tableofcontents
	\section*{Introduction}
	
	The study of symbolic powers of ideals has long been a central theme in Commutative Algebra. For an ideal $I$ in a Noetherian ring $R$, the $k$th symbolic power of $I$ is defined as
	\[
	I^{(k)}\ =\ \bigcap_{P\in\Ass(I)}(I^kR_P\cap R),
	\]
	where $\Ass(I)$ is the set of associated prime ideals of $I$.
	
	Symbolic powers first appeared in the literature in Krull’s proof of his principal ideal theorem \cite{Krull}. Subsequent work by Zariski \cite{Zar} and Nagata \cite[page 143]{Nag} highlighted their geometric significance, where they showed that if $P$ is a prime ideal of a polynomial ring $S$ over a field $K$, then one has $P^{(k)}=\bigcap\m^k$, where the intersection runs through all maximal ideals of $S$ containing $P$. This result was later refined by Eisenbud and Hochster \cite{EH}. More recently, De Stefani, Grifo and Jeffries extended these ideas to smooth $\mathbb{Z}$-algebras using $p$-derivations \cite{DGJ}. Moreover, the question of whether the symbolic Rees algebra $\mathcal{R}_s(I)=\bigoplus_{k\ge0}I^{(k)}t^k$ is Noetherian was answered in the negative by Roberts \cite{Ro}, following counterexamples to Hilbert's 14th problem by Nagata \cite{Nag1}. The containment problem between ordinary and symbolic powers has since been a major line of inquiry (see, e.g., \cite{DDGHB} and the references therein).
	
	Our primary motivation to study symbolic powers comes from combinatorics and the theory of monomial ideals. Over the past two decades, there has been an extensive investigation of symbolic powers of monomial ideals, beginning with the seminal work of Herzog, Hibi and Trung \cite{HHT} (see also \cite{BC,CDGR,DHNT,FMR,GHOS,MNPTV,MTr,MVu,SF0,SF,F2} and the references therein).
	
	Let $S=K[x_1,\dots,x_n]$ be the polynomial ring over a field $K$. In \cite{HHT}, it is shown that for any monomial ideal $I\subset S$, the symbolic Rees algebra  $\mathcal{R}_s(I)$ is Noetherian  and consequently, the Castelnuovo-Mumford regularity $\reg\,I^{(k)}$ is a quasi-linear function for $k\gg0$. By contrast, the regularity of the ordinary powers $\reg\,I^k$ is eventually linear (see \cite{CHT, K}). In the context of edge ideals, a remarkable conjecture of Minh \cite{MNPTV} posits that $\reg\,I(G)^{(k)}=\reg\,I(G)^k$ for all $k\ge1$, where $I(G)$ is the edge ideal of a finite simple graph $G$. If true, this would imply that $\reg\,I(G)^{(k)}$ is eventually linear, rather than merely quasi-linear, a surprising and impactful result.
	
	Among monomial ideals, polymatroidal ideals hold a distinguished place. These ideals are the algebraic counterpart of discrete polymatroids, a multiset analogue of matroids \cite{W} introduced by Herzog and Hibi \cite{HHD}. A monomial ideal $I\subset S$ is called \emph{polymatroidal} if the exponent vectors of its minimal generators form the set of bases of a discrete polymatroid. These ideals enjoy many desirable properties: for example, the product of polymatroidal ideals is again polymatroidal, they have linear resolutions, and consequently, all their powers have linear resolutions. Moreover, Bandari and Rahmati-Asghar \cite[Theorem 2.4]{BR} characterized polymatroidal ideals as those equigenerated monomial ideals that have linear quotients with respect to the lexicographic order induced by any ordering of the variables.
	
	However, not much is known about the symbolic powers of polymatroidal ideals. In this paper, our main objective is to address the following Conjectures A and B.
	\begin{Conjecture}\label{ConjA}
		Let $I\subset S$ be a polymatroidal ideal. Then 
		\begin{equation}\label{eq:PolyMinh}
			\reg\, I^{(k)}=\reg\, I^k \quad \text{for all } k\ge1.
		\end{equation}
	\end{Conjecture}
	If true, this would imply that $\reg\,I^{(k)}$ is eventually linear. Based on extensive computational evidence, we expect that every graded component of $I^{(k)}$ has a linear resolution, i.e., $I^{(k)}$ is \emph{componentwise linear} \cite{HH99}. We thus further raise the following
	\begin{Conjecture}\label{ConjB}
		Let $I\subset S$ be a polymatroidal ideal. Then $I^{(k)}$ is componentwise linear for all $k\ge1$.
	\end{Conjecture}
	If, in addition, the final degrees of $I^{(k)}$ and $I^k$ coincide (i.e., $\omega(I^{(k)})=\omega(I^k)$), then Conjecture~\ref{ConjA} would follow from Conjecture~\ref{ConjB}.
	
	The paper proceeds as follows. In Section \ref{sec1} we provide general bounds for the regularity $\reg\,I^{(k)}$ of symbolic powers of a monomial ideal $I\subset S$. In Theorem \ref{x-con-reg} we prove that $\reg\,I^{(k)}\le\reg_x\mathcal{R}_s(I)+d(I,k)$ for all $k\ge1$, where $d(I,k)$ is defined in terms of the bidegrees of the minimal monomial generators of the symbolic Rees algebra $\mathcal{R}_s(I)=\bigoplus_{k\ge0}I^{(k)}t^k$ and $\reg_x\mathcal{R}_s(I)$ is the so-called $x$-regularity of $\mathcal{R}_s(I)$. In \cite[Proposition 2.1]{HV}, Herzog and Vladoiu proved that any polymatroidal ideal can be expressed as the intersection of powers of monomial prime ideals. A monomial ideal with this property is said to be of \textit{intersection type}. We say that $I$ is of \textit{minimal intersection type} if $I$ is of intersection type and does not have embedded associated prime ideals. In Theorem \ref{Thm:reg-kk}, we prove that $\omega(I^{(k)})\ge\omega(I)k$ for all $k\ge1$, provided that $I$ is of minimal intersection type. Hence, if $I$ is a polymatroidal ideal without embedded associated primes, for instance if $I$ is matroidal, we obtain the inequality $\reg\,I^{(k)}\ge\reg\,I^k$ in the equation (\ref{eq:PolyMinh}) proposed in Conjecture \ref{ConjA}. In Theorem \ref{Thm:reg-kk}(d) we prove that if $I$ is a monomial ideal of minimal intersection type with linear powers such that $\reg_x\mathcal{R}_s(I)=0$, and $d(I,k)=\omega(I)k$ for all $k\ge1$, then $\reg\,I^{(k)}=\reg\,I^k$ for all $k\ge1$. We provide interesting families of monomial ideals $I$ satisfying $d(I,k)=\omega(I)k$ for all $k$, see Proposition \ref{Prop:dIk} and Corollaries \ref{Cor:perfect} and \ref{Cor:Idn-dIk}. 
	
	In Section \ref{sec2}, we study the symbolic Rees algebra $\mathcal{R}_s(I)$ of a monomial ideal $I$ of minimal intersection type. This algebra can be interpreted as the vertex cover algebra of a weighted simplicial complex introduced by Herzog, Hibi and Trung \cite{HHT}. The generators of $\mathcal{R}_s(I)$ correspond to the indecomposable $k$-covers of the simplicial complex associated to $I$. In general understanding the indecomposable $k$-covers is very challenging. We say that $I$ has a {\em linear cover function}, when the size of any indecomposable $k$-cover is a linear function of $k$ for all $k$. Squarefree Veronese ideals are examples of ideals with linear cover function. If $I$ has a linear cover function, in Theorem \ref{x-condition} we provide a criterion in terms of the defining ideal of $\mathcal{R}_s(I)$ and a special monomial order, called the \textit{$x$-condition}, which guarantees that $I^{(k)}$ has linear quotients, and in particular is componentwise linear for all $k\ge1$.  The concept of \textit{$x$-condition} was first exploited in~\cite{HHZ} and later in~\cite{HHM} and~\cite{HM}, in order to study the linearity and componentwise linearity of the ordinary powers of (monomial) ideals.
	
	In Section \ref{sec3} we apply the machinery developed in the previous sections to the class of squarefree Veronese ideals. These ideals have linear cover functions. In Theorem \ref{Veronese:x-condition} we prove that for any squarefree Veronese ideal $I$, the symbolic Rees algebra $\mathcal{R}_s(I)$ satisfies the $x$-condition. This result together with Theorem \ref{x-condition} implies that Conjectures \ref{ConjA} and \ref{ConjB} hold in such a case, see Corollary \ref{Cor:Veronese}. In Theorem \ref{gendeg} we determine the generating degrees of each symbolic power of a squarefree Veronese ideal $I$. As a consequence, we recover \cite[Theorem 7.5]{BC} on the Waldschmidt constant of $I$. Despite the fact that $\mathcal{R}_s(I)$ satisfies the $x$-condition for a squarefree Veronese ideal $I$, understanding the relations of the defining ideal of $\mathcal{R}_s(I)$ is challenging. However, in Proposition \ref{Prop:ini3} we show that, with respect to the monomial order (\ref{orderT}), the initial ideal of the defining ideal of $\mathcal{R}_s(I)$ is generated in degree at most 3. 
	
	A monomial ideal $I\subset S$ is called \textit{componentwise polymatroidal} if $I_{\langle j\rangle}$ is polymatroidal for all $j$. By \cite[Theorem 3.1]{Fic0} componentwise polymatroidal ideals have linear quotients. Naively, one may expect that $I^{(k)}$ is componentwise polymatroidal if $I$ is polymatroidal. This is not the case in general as we show at page \pageref{notcp}. 
	
	Section \ref{sec4} is devoted to matching-matroidal ideals. A famous theorem of Edmonds and Fulkerson shows that a matroid is a matching matroid if and only if it is transversal, see \cite[Theorem 2 on page 248]{W}. In Theorem \ref{Thm:mm} we prove that an ideal $I$ is matching-matroidal if and only if it is the squarefree part of the product of some monomial prime ideals. For the subfamily of matching-matroidal ideals of Veronese type, we prove in Theorem \ref{Thm:mat-mat} that Conjectures \ref{ConjA} and \ref{ConjB} hold.
	
	In Section \ref{sec5} we address the problem of classifying the polymatroidal ideals $I\subset S$ whose ordinary and symbolic powers coincide, i.e., $I^{(k)}=I^k$ for all $k\ge1$. First we consider the squarefree case. That is, we consider matroidal ideals. A famous conjecture posed by Conforti and Cornu\'ejols \cite{CC} has been restated equivalently in algebraic terms by Gitler, Villarreal and others \cite{GRV,GVV} as follows.
	
	\begin{Conjecture}\label{Conj-CC}
		A squarefree monomial ideal $I\subset S$ satisfies $I^{(k)}=I^k$ for all $k\ge1$ if and only if $I$ is packed.
	\end{Conjecture}
	
	For the definition of packed ideals see Section \ref{sec5}. Conjecture \ref{Conj-CC} is widely open. It is known that it holds true for any edge ideal \cite{GRV,GVV}. In Theorem \ref{Thm:I(k)k-sq} we establish Conjecture \ref{Conj-CC} for all matroidal ideals. We prove that a matroidal ideal $I$ is packed if and only if $I$ is the product of monomial prime ideals with pairwise disjoint supports.
	
	Then, we turn to the problem of classifying non-squarefree polymatroidal ideals $I$ whose ordinary and symbolic powers coincide. In Lemma \ref{Lem:el} 
	we show that $I^{(k)}=I^k$ for all $k$, provided that $\m=(x_1,\dots,x_n)\in\Ass(I)$. Using this elementary result, we are able to construct in a simple fashion several polymatroidal ideals whose ordinary and symbolic powers coincide, see Proposition \ref{Prop:new-k-k}. In Theorem \ref{Thm:k(k)}, we prove that the following families of polymatroidal ideals satisfy $I^{(k)}=I^k$ for all $k\ge1$:
	\begin{enumerate}
		\item[(a)] Polymatroidal ideals generated in degree $2$ which are not squarefree.
		\item[(b)] Transversal polymatroidal ideals.
		\item[(c)] Principal Borel ideals.
	\end{enumerate}
	
	In Section \ref{sec6} utilizing the results in the paper, we present several families of polymatroidal ideals for which we are able to verify Conjectures \ref{ConjA} and \ref{ConjB}. Using results from \cite{BR,FMR} we prove in Proposition \ref{Prop:I(G)-reg} that Conjectures \ref{ConjA} and \ref{ConjB} hold for any polymatroidal ideal generated in degree 2. In this vein, we establish these conjectures for all polymatroidal ideals in up to 3 variables (Proposition \ref{Prop:three}), and for all matroidal ideals in at most 4 variables (Proposition \ref{Prop:four}). 
	
	In view of the results in this paper, \cite[Conjecture B]{FMR} and the results in \cite{FMR}, and \cite[Theorem 3.6 and Corollary 3.7]{SF0} we are tempted to ask the following\medskip
	
	\noindent\textbf{Question D.}\,\, \textit{Let $I\subset S$ be an equigenerated monomial ideal whose all powers have linear resolution $($or linear quotients$)$. Is it true that $\reg\,I^{(k)}=\reg\,I^k$ and $I^{(k)}$ is componentwise linear, $($or has linear quotients$)$, for all $k\ge1$ ?}\medskip
	
	There are nice families of monomial ideals $I$ for which the $x$-regularity $\reg_x\mathcal{R}(I)$ of the ordinary Rees algebra $\mathcal{R}(I)=\bigoplus_{k\ge0}I^kt^k$ is zero, for instance if $I$ is an edge ideal with linear resolution, see \cite{HHZ}. Therefore, in view of Theorem \ref{Thm:reg-kk}(d), we are led to speculate whether in general we have $\reg_x\mathcal{R}_s(I)\le\reg_x\mathcal{R}(I)$, or to determine under which conditions this inequality holds.
	
	\section{General bounds for the regularity of symbolic powers}\label{sec1}
	
	In this section, we bound the regularity of symbolic powers of monomial ideals, and prove the inequality $\reg\,I^{(k)}\ge\reg\,I^k$ for any polymatroidal ideal $I$ without embedded associated primes.
	
	Throughout this paper $S=K[x_1,\dots,x_n]$ is the standard graded polynomial ring over a field $K$ and $\m=(x_1,\dots,x_n)$ is the unique graded maximal ideal of $S$. For a graded ideal $I\subset S$, we denote by $\alpha(I)=\min\{d:\ (I/\m I)_d\ne0\}$ the \textit{initial degree} of $I$ and by $\omega(I)=\max\{d:\ (I/\m I)_d\ne0\}$ the \textit{final degree} of $I$.
	
	Firstly, we provide a general upper bound for the regularity $\reg\,I^{(k)}$ of symbolic powers of a monomial ideal.
	
	Let $I\subset S$ be a monomial ideal. By \cite[Theorem 3.2]{HHT}, the symbolic Rees algebra $\mathcal{R}_s(I)=\bigoplus_{k\ge0}I^{(k)}t^k$ is a toric ring with minimal monomial generators $x_1,\dots,x_n,u_1t^{q_1},\dots,u_mt^{q_m}$, with $u_i\in S$ monomials and $q_i>0$, for all $i=1,\dots,m$.
	
	Let $T=K[x_1,\dots,x_n,y_1,\dots,y_m]$ be the bigraded polynomial ring over $K$ with the bigrading given by $\bideg(x_i)=(1,0)$ for $i=1,\dots,n$ and $\bideg(y_j)=(\deg(u_j),q_j)$ for $j=1,\dots,m$. Then $\mathcal{R}_s(I)\cong T/\Ker\,\varphi$, where $\varphi:T\rightarrow\mathcal{R}_s(I)$ is the $K$-algebra homomorphism defined by setting $\varphi(x_i)=x_i$ and $\varphi(y_j)=u_jt^{q_j}$ for all $i=1,\dots,n$ and $j=1,\dots,m$. Notice that $\mathcal{R}_s(I)$ is a finitely generated bigraded $T$-module. Let
	\begin{equation}\label{eq:F}
		\FF:0\rightarrow F_p\rightarrow\cdots\rightarrow F_1\rightarrow F_0\rightarrow\mathcal{R}_s(I)\rightarrow0    
	\end{equation}
	be the minimal bigraded $T$-resolution of $\mathcal{R}_s(I)$. We have $F_i=\bigoplus_j T(-a_{ij},-b_{ij})$ for all $i=0,\dots,p$. We define the \textit{$x$-regularity} of $\mathcal{R}_s(I)$ as the integer
	$$
	\reg_x\mathcal{R}_s(I)\ =\ \max_{i,j}\{a_{ij}-i\}.
	$$
	
	\begin{Theorem}\label{x-con-reg}
		We keep the notation as above. For all $k\ge1$, set
		$$
		d(I,k)=\max_{\substack{{\bf c}=(c_1,\dots,c_m)\in\ZZ_{\ge0}^m\\ c_1q_1+\dots+c_mq_m\le k}}\{\sum_{s=1}^m c_s\deg(u_s)\}.
		$$
		Then, for all $k\ge1$ we have $$\reg\,I^{(k)}\le \reg_x\mathcal{R}_s(I)+d(I,k).$$
	\end{Theorem}
	\begin{proof}
		For a bigraded $T$-module $M=\bigoplus_{d,k\ge0}M_{(d,k)}$ we put $M_{(*,k)}=\bigoplus_{d\ge0}M_{(d,k)}$. Since $\mathcal{R}_s(I)_{(*,k)}\cong I^{(k)}$, the complex (\ref{eq:F}) induces an exact complex
		\begin{equation}\label{eq:Fk}
			\FF_k:0\rightarrow (F_p)_{(*,k)}\rightarrow\cdots\rightarrow (F_1)_{(*,k)}\rightarrow (F_0)_{(*,k)}\rightarrow I^{(k)}\rightarrow0
		\end{equation}
		which is a free $S$-resolution of $I^{(k)}$. Next, since $\bideg(y_s)=(\deg(u_s),q_s)$ for all $s$, we have the isomorphism of free $S$-modules
		\begin{equation}\label{eq:Fc}
			(F_i)_{(*,k)}\cong\bigoplus_j\bigoplus_{\substack{{\bf c}=(c_1,\dots,c_m)\in\ZZ_{\ge0}^m\\ c_1q_1+\dots+c_mq_m=k-b_{ij}}}S(-a_{ij}){\bf y^c},
		\end{equation}
		where ${\bf y^c}=y_1^{c_1}\cdots y_m^{c_m}$. After shifting in (\ref{eq:Fc}) each ${\bf y^c}$ by its first bidegree component, it follows that (\ref{eq:Fk}) is a graded free $S$-resolution of $I^{(k)}$. This yields at once that
		$$
		\reg\,I^{(k)}\le\max_{i,j}\{a_{ij}-i\}+\max_{\substack{{\bf c}=(c_1,\dots,c_m)\in\ZZ_{\ge0}^m\\ c_1q_1+\dots+c_mq_m\le k}}\{\sum_{s=1}^m c_s\deg(u_s)\}=\reg_x\mathcal{R}_s(I)+d(I,k),
		$$
		as desired.
	\end{proof}
	
	Given a monomial ideal $I\subset S$ we denote by $\mathcal{G}(I)$ the unique set of minimal monomial generators of $I$. Moreover,  we put $[n]=\{1,\dots,n\}$ for an integer $n\ge1$. For a non-empty subset $F$ of $[n]$ we put $P_F=(x_i:\ i\in F)$.
	
	Let $I\subset S$ be a monomial ideal and let $P\in\Spec(S)$ be a monomial prime ideal. The \textit{monomial localization} of $I$ with respect to $P$ is the monomial ideal $I(P)$ of the polynomial ring $S(P)=K[x_i:\ x_i\in P]$ obtained by applying the substitutions $x_i\mapsto 1$ for all $x_i\notin P$. Since $I^kS_P\cap S=(I(P))^kS_P\cap S=(I(P))^k$, it follows that
	\begin{equation}\label{eq:symbolicP'}
		I^{(k)}\ =\ \bigcap_{P\in\Ass(I)} (I(P))^k.
	\end{equation}
	
	Recall that the \textit{support} of a monomial $u\in S$ is the set $\supp(u)=\{i:\ x_i\ \textup{divides}\ u\}$ and the \textit{support} of a monomial ideal $I\subset S$ is the set $\supp(I)=\bigcup_{u\in\mathcal{G}(I)}\supp(u)$. We say that $I\subset S$ is \textit{fully-supported} if $\supp(I)=[n]$.
	
	A monomial ideal $I\subset S$ is called \textit{polymatroidal} if the exponent vectors of the minimal monomial generators of $I$ form the set of bases of a discrete polymatroid. A squarefree polymatroidal ideal is called \textit{matroidal}.
	
	For a monomial $u=x_1^{a_1}\cdots x_n^{a_n}\in S$, we define the \textit{$x_i$-degree} of $u$ as the integer $\deg_{x_i}(u)=\max\{j:\ x_i^j\mid u\}=a_i$. Polymatroidal ideals may as well be characterized by the so-called \textit{exchange property}. Hence, a monomial ideal $I\subset S$ is polymatroidal if and only if $I$ is equigenerated and for all $u,v\in\mathcal{G}(I)$ with $\deg_{x_i}(u)>\deg_{x_i}(v)$ for some $i$, there exists $j$ with $\deg_{x_j}(u)<\deg_{x_j}(v)$ such that $x_j(u/x_i)\in\mathcal{G}(I)$.\smallskip
	
	Following \cite{HV}, we say that a monomial ideal $I\subset S$ is of \textit{intersection type} if $I$ is the intersection of powers of monomial prime ideals. Polymatroidal ideals are of intersection type \cite[Proposition 2.1]{HV}.
	
	We say that $I$ is of \textit{minimal intersection type} if $I$ is of intersection type and has no embedded associated prime ideals. Hence $I$ is of the form $I=\bigcap_{i=1}^m P_i^{k_i}$, where the $P_i$'s are monomial prime ideals and $P_i\nsubseteq P_j$ for all $i\neq j$.\smallskip
	
	Notice that $d(I,k)\ge\omega(I)k$ for any monomial ideal $I\subset S$. Indeed, among the minimal generators of $\mathcal{R}_s(I)$ we have $ut$ for any $u\in\mathcal{G}(I)$ with $\deg(u)=\omega(I)$.\smallskip
	
	For a monomial $u=x_1^{a_1}\cdots x_n^{a_n}\in S$ and a subset $F\subseteq[n]$, we put $u_F=\prod_{i\in F}x_i^{a_i}$.
	\begin{Theorem}\label{Thm:reg-kk}
		Let $I\subset S$ be a monomial ideal of minimal intersection type. Then
		\begin{enumerate}
			\item[(a)] $u^k\in\mathcal{G}(I^{(k)})$ for any $u\in \mathcal{G}(I)$, for all $k\ge1$.
			\item[(b)] $\reg\,I^{(k)}\ge\omega(I^{(k)})\ge\omega(I)k$, for all $k\ge1$.
			\item[(c)] If $I$ has linear powers, then  
			$\reg\,I^{(k)}\ge\reg\,I^{k}$, for all $k\ge1$.
			\item[(d)] If $I$ has linear powers, $\reg_x\mathcal{R}_s(I)=0$ and $d(I,k)=\omega(I)k$ for all $k\ge1$, then $\reg\,I^{(k)}=\reg\,I^{k}$ for all $k\ge1$.
		\end{enumerate}
	\end{Theorem}
	\begin{proof}
		(a) By assumption, $I=P_1^{k_1}\cap\dots\cap P_m^{k_m}$ where $P_i\not\subseteq P_j$ for all integers $i\ne j$. In particular, $\Ass(I)=\{P_1,\dots,P_m\}$. Since $I$ has no embedded primes, we obtain that $(I(P_i))^k=P_i^{kk_i}$ for all $i$. Hence, by (\ref{eq:symbolicP'}) we have
		\begin{equation}\label{eq:Ik-mit}
			I^{(k)}=P_1^{kk_1}\cap\dots\cap P_m^{kk_m}.
		\end{equation}
		
		Now, choose a monomial $u\in\mathcal{G}(I)$. Notice that $u^k\in I^k\subseteq I^{(k)}$. As $u\in\mathcal{G}(I)$, for every $i\in\supp(u)$, there exists $1\le j\le m$ such that $u/x_i\notin P_j^{k_j}$.
		
		Choose $i\in\supp(u)$ and let $v=u/x_i$. Then $v\notin P_j^{k_j}$ for some $j$. We have $P_j=P_F$ for some $F\subseteq[n]$. It follows that $\deg(v_F)\le k_j-1$. We claim that $w=u^k/x_i=v^kx_i^{k-1}\notin P_j^{kk_j}$. Notice that $\deg(w_F)\le\deg(v_F^k)+k-1$. Thus
		$$
		\deg(w_F)\le\deg(v_F)k+k-1\le(k_j-1)k+k-1=kk_j-1.
		$$
		This shows that $w=u^k/x_i\notin P_j^{kk_j}$. Equation (\ref{eq:Ik-mit}) implies that $u^k/x_i\notin I^{(k)}$, for every $i\in\supp(u)$. Therefore $u^k\in\mathcal{G}(I^{(k)})$.
		
		(b) Let $u\in\mathcal{G}(I)$ be a monomial with $\deg(u)=\omega(I)$. By (a) we have $u^k\in\mathcal{G}(I^{(k)})$. Hence, $\reg\,I^{(k)}\ge\omega(I^{(k)})\ge\deg(u^k)=\omega(I)k$.
		
		(c)	 We have $\reg\,I^k=\omega(I^k)=\omega(I)k$ for all $k\ge1$. Hence, the assertion follows from (b) because $\reg\,I^{(k)}\ge\omega(I^{(k)})\ge\omega(I^k)$ for all $k\ge1$.
		
		(d) Since $I$ has linear powers, we have $\reg\,I^k=\alpha(I)k=\omega(I)k$ for all $k\ge1$. Using part (b), Theorem \ref{x-con-reg}(a) and the assumptions $\reg_x\mathcal{R}_s(I)=0$ and $d(I,k)=\omega(I)k$ for all $k\ge1$, we have $\reg\,I^{(k)}=\omega(I)k=\reg\,I^k$ for all $k\ge1$, as desired.
	\end{proof}
	
	As a consequence, we establish the inequality $\reg\,I^{(k)}\ge\reg\,I^k$ of the equation proposed in Conjecture \ref{ConjA} for any polymatroidal ideal $I\subset S$ without embedded associated primes. Such an ideal is of minimal intersection type and has linear powers. Hence, Theorem~\ref{Thm:reg-kk}(c) implies
	
	\begin{Corollary}\label{unmixedPoly}
		Let $I\subset S$ be a polymatroidal ideal without embedded associated primes. Then for all $k\ge1$,
		$$\reg\,I^{(k)}\ge\reg\,I^{k}.
		$$
	\end{Corollary}
	
	One may wonder whether the condition $d(I,k)=\omega(I)k$ given in Theorem \ref{Thm:reg-kk}(d) is met by interesting families of monomial ideals. In this direction, we have
	\begin{Proposition}\label{Prop:dIk}
		Let $I\subset S$ be a monomial ideal such that for any minimal monomial generator $ut^q$ of $\mathcal{R}_s(I)$ with $q>0$ and $u\in S$ we have $\deg(u)\le q+\omega(I)-1$. Then, for all $k\ge1$ we have
		$$
		d(I,k)=\omega(I)k.
		$$
	\end{Proposition}
	\begin{proof}
		The assumption implies that $\mathcal{R}_s(I)=S[\{u_{i,1}t^i,\dots,u_{i,s_i}t^i\}_{i=1,\dots,r}]$, where the $u_{i,j}$'s are monomials of $S$ of degree at most $i+\omega(I)-1$. Hence,
		\begin{equation}\label{eq:d(I,k)}
			\omega(I)k\le d(I,k)\le\max_{\bf c}\{\sum_{i=1}^r\sum_{j=1}^{s_i}(i+\omega(I)-1)c_{ij}\},
		\end{equation}
		where the maximum is taken over all vectors ${\bf c}=(c_{11},\dots,c_{1s_1},\dots,c_{r1},\dots,c_{rs_r})$ such that $\sum_{i=1}^r\sum_{j=1}^{s_i}c_{ij}i\le k$. We put $c_i=\sum_{j=1}^{s_i}c_{ij}$ for all $i=1,\dots,r$. Then, we have
		\begin{equation}\label{eq:c1}
			c_1+2c_2+\dots+rc_r\le k.
		\end{equation}
		Hence $c_1+c_2+\dots+c_r\le k$ and so 
		\begin{equation}\label{eq:c2}
			(\omega(I)-1)c_1+\dots+(\omega(I)-1)c_r\le (\omega(I)-1)k
		\end{equation}
		Summing (\ref{eq:c1}) and (\ref{eq:c2}) and comparing with (\ref{eq:d(I,k)}), the assertion follows.
	\end{proof}	
	
	We close this section by providing an interesting family of monomial ideals $I$ satisfying the condition $d(I,k)=\omega(I)k$ for all $k\ge1$.
	
	For a finite simple graph $G$ with the vertex set $\{x_1,\dots,x_n\}$ and the edge set $E(G)$, the \textit{edge ideal} of $G$ is defined as the ideal $I(G)=(x_ix_j:\ \{x_i,x_j\}\in E(G))$.
	
	By \cite[Corollary 3.3]{Vil} the edge ideal $I=I(G)$ of any perfect graph $G$ satisfies the assumption of Proposition \ref{Prop:dIk}. Hence
	\begin{Corollary}\label{Cor:perfect}
		Let $G$ be a perfect graph and let $I=I(G)$. Then $d(I,k)=2k$ for all $k\ge1$.
	\end{Corollary}
	
	\section{The $x$-condition for symbolic Rees algebras}\label{sec2}
	
	In this section we study the symbolic powers of monomial ideals of minimal intersection type by investigating the defining ideal of their symbolic Rees algebra. We provide a criterion on monomial ideals of minimal intersection type which implies componentwise linearity for all symbolic powers of such ideals.
	
	A homogeneous ideal $I\subset S$ is called {\em componentwise linear}, if the ideal $$I_{\langle d\rangle}=(f\in I:\, f \textrm{ is homogeneous of degree } d)$$ has linear resolution for any positive integer $d$. If $I\subset S$ is componentwise linear, then $\reg\,I=\omega(I)$. An useful approach to show that an ideal is componentwise linear is to show that it has linear quotients. Indeed, monomial ideals with linear quotients are componentwise linear \cite[Theorem 8.2.15]{HHBook}.
	
	Recall that $I$ has {\em linear quotients} if the set $\mathcal{G}(I)$ can be ordered as $u_1,\ldots,u_m$ such that for each $i = 2,\ldots,m$, the ideal $(u_1,\ldots,u_{i-1}):(u_i)$ is generated by variables.
	In the following, for two monomials $u$ and $v$, we set $u:v=u/\gcd(u,v)$. Notice that
	$$(u_1,\ldots,u_{i-1}):(u_i)=(u_1:u_i,\dots,u_{i-1}:u_i).$$
	
	Now, let $I=\bigcap_{i=1}^m P_i^{k_i}$ be a monomial ideal of minimal intersection type. Let $\Delta(I)$ be the simplicial complex whose facets are $F_i=\supp(P_i)$, $i=1,\ldots,m$. Then the symbolic Rees algebra $\mathcal{R}_s(I)$ is the vertex cover algebra $A(\Delta(I),w)$ of the weighted simplicial complex $(\Delta(I),w)$, where $w:\mathcal{F}(\Delta(I))\to \ZZ_{>0}$ is the (weight) function $w(F_i)=k_i$ for all $i$, see~\cite[Lemma 4.1]{HHT}. Here $\mathcal{F}(\Delta(I))$ denotes the set of maximal faces of $\Delta(I)$. The {\em vertex cover algebra} 
	$A(\Delta,w)$ of a weighted simplicial complex $\Delta$ with the weight function $w:\mathcal{F}(\Delta)\to \ZZ_{>0}$  was introduced in~\cite{HHT}, as the $S$-algebra generated by all monomials 
	${\bf x^a}t^k=x_1^{a_1}\cdots x_n^{a_n}t^k$, such that 
	${\bf a} =(a_1\ldots,a_n)\in \ZZ_{\ge0}^n$ satisfies
	\begin{equation}\label{k-cover}
		\sum_{i\in F} a_i\ge kw(F)\ \ \ \textup{for all } F\in\mathcal{F}(\Delta).  
	\end{equation}
	
	A vector ${\bf a}$ satisfying~(\ref{k-cover}) is called a {\em vertex cover} of $(\Delta,w)$ of order $k$ or simply a $k$-cover of $(\Delta,w)$.
	A $k$-cover ${\bf a}$ is called  {\em indecomposable}, if one can not write ${\bf a}$ as  ${\bf a}={\bf a}_1+{\bf a}_2$, where ${\bf a}_1$ is a $k_1$-cover and ${\bf a}_2$ is a $k_2$-cover of $(\Delta,w)$ with $k=k_1+k_2$. 
	Clearly, the set of the indecomposable covers of $(\Delta,w)$ corresponds to the (unique) minimal monomial generating set of $A(\Delta,w)$ as a $S$-algebra.
	
	For a vector ${\bf a}=(a_1\ldots,a_n)$ we set $|{\bf a}|=a_1+\cdots+a_n$. We say that $I$ has a {\em linear cover function}, if there are integers $c,d\ge 0$ such that $|{\bf a}|=ck+d$ for any indecomposable $k$-cover ${\bf a}$ of $(\Delta(I),w)$  and any $k$. If this is the case, then we say that $\nu_I(k)=ck+d$ is the (linear) {\em cover function} of $I$.
	
	By~\cite[Theorem 3.2]{HHT}, $\mathcal{R}_s(I)$ is a finitely generated $S$-algebra. The minimal monomial generators $w_1,\ldots,w_m$ of the $S$-algebra $\mathcal{R}_s(I)$ are the monomials ${\bf x^a}t^k$ for which ${\bf a}$ is an indecomposable $k$-cover of $(\Delta(I),w)$. Let $T=S[y_1,\dots,y_m]$ be a polynomial ring and let $\varphi:T\rightarrow\mathcal{R}_s(I)$ be the $S$-algebra map defined by $\varphi(y_i)=w_i$ for all $i=1,\dots,m$.
	
	Let $J=\Ker\,\varphi$. We say that $\mathcal{R}_s(I)$ satisfies the $x$-condition with respect to a monomial order $>$ on $T$ if  any minimal monomial generator of $\ini_<(J)$ is of the form $vv'$ with $v\in S$ of degree at most one and $v'\in K[y_1,\ldots,y_m]$.
	
	We fix an order on the generators $w_1>\dots>w_m$ of $\mathcal{R}_s(I)$ as follows:
	\begin{equation}\label{orderY}\textup{$ut^p>vt^q$ $\Longleftrightarrow$     $p>q$ or $p=q$ and $u>_\lex v$},
	\end{equation}
	where $>_\lex$ is the lex order on $S$ induced by $x_1>\dots>x_n$. Let $>_{\rlex}$ be the reverse lex order induced by $y_1>\dots>y_m$. We fix a monomial order $>$ on $T$ as follows: 
	\begin{equation}\label{orderT}
		(\prod_{i}y_i^{a_i})(\prod_{j}x_j^{b_j})>(\prod_{i}y_i^{a_i'})(\prod_{j}x_j^{b_j'})
	\end{equation}
	if and only if
	\begin{enumerate}
		\item[(i)] either $\prod_{i}y_i^{a_i}>_{\rlex}\prod_{i}y_i^{a_i'}$,
		\item[(ii)] or else $\prod_{i}y_i^{a_i}=\prod_{i}y_i^{a_i'}$ and $\prod_{j}x_j^{b_j}>_{\lex}\prod_{j}x_j^{b_j'}$.
	\end{enumerate}\medskip
	
	The following lemma is required for the proof of Theorem \ref{x-condition}.
	\begin{Lemma}\label{cover.function}
		Let $I\subset S$ be a monomial ideal of minimal intersection type which has a linear cover function, say $\nu_I(k)=ck+d$. Let $g\in R=K[y_1,\ldots,y_m]$ be a monomial, and $k>0$ be the integer with $\varphi(g)\in I^{(k)}t^k$. Then $\deg(\varphi(g))=(c+1)k+\deg(g)d$.
	\end{Lemma}
	\begin{proof}
		Let $g=y_{i_1}\cdots y_{i_s}$, and $\varphi(y_{i_j})=
		\textbf{x}^{\textbf{a}_j} t^{k_j}$ for indecomposable $k_j$-covers $\textbf{a}_j$ of $(\Delta(I),w)$, and integers $k_j>0$,  $j=1,\ldots,s$. Then $\varphi(g)=\textbf{x}^{\textbf{a}_1+\cdots+\textbf{a}_s}  t^{k_1+\cdots+k_s}$, and from $\varphi(g)\in I^{(k)}t^k$ we obtain $k_1+\cdots+k_s=k$. Moreover, $|\textbf{a}_j|=\nu_I(k_j)=ck_j+d$ for all $j$.  Hence, $$\deg(\varphi(g))=k+\sum_{j=1}^s |\textbf{a}_j|=k+\sum_{j=1}^s (ck_j+d)=k+ck+sd=(c+1)k+\deg(g)d,$$
		as desired.
	\end{proof}
	
	The next result gives a criteria for the componentwise linearity of symbolic powers of monomial ideals of minimal intersection type.
	
	\begin{Theorem}\label{x-condition}
		Let $I\subset S$ be a monomial ideal of minimal intersection type which has a linear cover function. If  $\mathcal{R}_s(I)$ satisfies the $x$-condition with respect to the order given in \textup{(\ref{orderT})}, then $I^{(k)}$ has linear quotients for all $k\ge1$. In particular, $I^{(k)}$ is componentwise linear for all $k$. 
	\end{Theorem}
	\begin{proof}
		We adopt the notation introduced above. For a minimal monomial generator $u$ of $I^{(k)}$, the element $ut^k\in A(\Delta(I),w)$ has a presentation of the form $ut^k=(\textbf{x}^{\textbf{a}_1}t^{q_1})\cdots (\textbf{x}^{\textbf{a}_s}t^{q_s})$ such that each $\textbf{a}_j$ is an indecomposable $q_j$-cover of $(\Delta(I),w)$, $q_j>0$, and $q_1+\cdots+q_s=k$. Such presentation is the image of a monomial $y_{i_1}\cdots y_{i_s}$ in $R=K[y_1,\ldots,y_m]$ under the map $\varphi$ with $\varphi(y_{i_\ell})=\textbf{x}^{\textbf{a}_\ell} t^{q_\ell}$ for all $\ell$. Then, there exists a presentation among all presentations of $ut^k$ in $A(\Delta(I),w)$, whose preimage in $R$ is the smallest term with respect to $>_{\rlex}$ among the preimages of all presentations of $ut^k$. Indeed, the relations $q_1+\cdots+q_s=k$ and $q_i>0$ for all $i$ show that any monomial in $R$ which is mapped to $ut^k$ is of degree at most $k$. So the number of monomials in $R$ which are mapped to $ut^k$ is finite and hence there is a minimal one among them with respect to $>_{\rlex}$. We denote this monomial by $f_u$.
		
		First we show that $f_u\notin\ini_<(J)$ for any $u\in\mathcal{G}(I^{(k)})$. 
		Assume on the contrary that there exists a relation $h=f_u-vg\in J$, with $\ini_<(h)=f_u$, where $v\in S$ and $g\in R$ are monomials. Then $ut^k=\varphi(f_u)=v\varphi(g)$ and $\varphi(g)\in I^{(k)}t^k$. Since $u$ is a minimal generator of $I^{(k)}$, we conclude that $v=1$. Then $h=f_u-g\in J$ which means that $\varphi(g)=\varphi(f_u)=ut^k$. Hence $g$ is also the preimage of $ut^k$ under $\varphi$. Then by our choice of $f_u$ we get $g>_{\rlex} f_u$. This contradicts to $\ini_<(h)=f_u$. Consequently $f_u\notin\ini_<(J)$, as desired. 
		
		We put an order on the minimal set of monomial generators of $I^{(k)}$ as follows:
		for $u,u'\in \mathcal{G}(I^{(k)})$ we set $u<u'$ if and only if $f_u<_{\rlex} f_{u'}$. We show that 
		$I^{(k)}$ has linear quotients with respect to this order on its minimal generating set. Let $u_1<\cdots<u_r$ be such an order on $\mathcal{G}(I^{(k)})$. Then $f_{u_1}<_{\rlex}\cdots <_{\rlex} f_{u_r}$. Fix an integer $2\le i\le r$ and let $v\in (u_1,\ldots,u_{i-1}): (u_i)$ be a monomial. Then $vu_i=v'u_j$ for some $j<i$ and a monomial $v'\in S$. Hence, $p=vf_{u_i}-v'f_{u_j}\in J$ with $\ini_<(p)=vf_{u_i}$. By our assumption that $\mathcal{R}_s(I)$ satisfies the $x$-condition, there exists a monomial $v_1g_1\in \ini_<(J)$ with $v_1\in S$ of degree at most one and $g_1\in R$ such that $v_1g_1\mid vf_{u_i}$. First we show that $v_1\neq 1$. Indeed, if $v_1=1$, then $g_1\in \ini_<(J)$ and since $g_1\mid f_{u_i}$, we get $f_{u_i}\in \ini_<(J)$, which is not possible, as shown above. 
		Hence, $v_1\neq 1$. Since $v_1$ is of degree at most one, we obtain $v_1=x_{\ell}$ for some $\ell$. So $x_{\ell}g_1\in \ini_<(J)$. Let $p'=x_{\ell}g_1-v_2g_2\in J$ be a binomial with $\ini_<(p')=x_{\ell}g_1$, where $v_2\in S$ and $g_2\in R$ are monomials. From $g_1\mid f_{u_i}$ we have $f_{u_i}=g_1g_3$ with $g_3\in R$.  Therefore, $g_3p'=x_\ell f_{u_i}-v_2g_2g_3\in J$. Since $g_1>_{\rlex} g_2$ we have $f_{u_i}>_{\rlex} g_2g_3$. 
		We consider the following two cases.\smallskip
		
		\textbf{Case 1.}
		Suppose that $g_2g_3$ is mapped to a minimal monomial generator, say $u_st^k$ of $I^{(k)}t^k$. Then $g_2g_3\ge _{\rlex} f_{u_s}$. From this and $f_{u_i}>_{\rlex} g_2g_3$ we get $f_{u_i}>_{\rlex} f_{u_s}$. The definition of our order on $\mathcal{G}(I^{(k)})$  implies that $u_s<u_i$. Moreover, $$x_\ell u_it^k=\varphi(x_{\ell}f_{u_i})=\varphi (v_2g_2g_3)=v_2\varphi (g_2g_3)=v_2u_st^k.$$ Hence $x_\ell \in (u_1,\ldots,u_{i-1}):(u_i)$ with $x_\ell$ dividing $v$, as desired.\smallskip
		
		\textbf{Case 2.} Suppose that $\varphi(g_2g_3)$ is not a minimal generator of $I^{(k)}t^k$. Then $\varphi(g_2g_3)=v''(u_st^k)$ for some monomial $v''\in S$ with $v''\neq 1$ and some $1\leq s\leq r$. Then $g_2g_3-v''f_{u_s}\in J$. We show that $g_2g_3>_{\rlex} f_{u_s}$. To this aim it is enough to show that $\deg(f_{u_s})<\deg(g_2g_3)$. Suppose that this is not the case and $\deg(f_{u_s})\ge \deg(g_2g_3)$. Since $I$ has a linear cover function, by Lemma~\ref{cover.function}, there exist integers $c,d\ge 0$ such that $\deg(u_st^k)=(c+1)k+\deg(f_{u_s})d$ and $\deg(\varphi(g_2g_3))=(c+1)k+\deg(g_2g_3)d$. Hence, $\deg(u_st^k)\ge \deg(\varphi(g_2g_3))$.
		This contradicts to $v''\neq 1$.  Therefore, $g_2g_3>_{\rlex} f_{u_s}$. From $f_{u_i}>_{\rlex} g_2g_3$ we obtain $f_{u_i}>_{\rlex} f_{u_s}$ and hence $u_s<u_i$. Moreover,  $$x_\ell f_{u_i}-v_2v''f_{u_s}=(x_\ell f_{u_i}-v_2g_2g_3)+v_2(g_2g_3-v''f_{u_s})\in J.$$ So we get 
		$$x_\ell u_it^k=\varphi(x_{\ell}f_{u_i})= \varphi (v_2v''f_{u_s})=v_2v''u_st^k.$$ Hence, $x_\ell u_i=v_2v''u_s$ with $s<i$. Therefore, $x_\ell \in (u_1,\ldots,u_{i-1}):(u_i)$ with $x_\ell$ dividing $v$.
	\end{proof}\medskip
	
	\section{Symbolic powers of squarefree Veronese ideals}\label{sec3}
	
	The goal of this section is to prove Conjectures \ref{ConjA} and \ref{ConjB} for any squarefree Veronese ideal.
	Let $n$ and $d$ be positive integers. The \textit{squarefree Veronese ideal} of degree $d$ of $S=K[x_1,\ldots,x_n]$ is the monomial ideal $I_{n,d}$ generated by all squarefree monomials of degree $d$ in $S$,
	$$
	I_{n,d}\ =\ (x_{i_1}x_{i_2}\cdots x_{i_d}\ :\ 1\le i_1<i_2<\dots<i_d\le n).
	$$
	
	Notice that $I_{n,d}=(0)$, whenever $d>n$. We remark that $I_{n,d}$ is a polymatroidal ideal of minimal intersection type. 
	
	Let $I=I_{n,d}$. By~\cite[Lemma 4.1, Proposition 4.6]{HHT} it follows that
	
	\begin{equation}\label{eq:Rees}
		\mathcal{R}_s(I)=S[x_{i_1}\cdots x_{i_{q+d-1}}t^q:\ 1\leq q\leq n-d+1,\ i_1<\cdots<i_{d+q-1} ].
	\end{equation}
	
	It is easy to see from (\ref{eq:Rees}) that a vector $\textbf{a}\in \ZZ_{\ge0}^n$ is an indecomposable $q$-cover of $(\Delta(I),w)$ if and only if $\textbf{x}^{\textbf{a}}$ is squarefree of degree $q+d-1$. Hence, Proposition \ref{Prop:dIk} implies immediately
	\begin{Corollary}\label{Cor:Idn-dIk}
		Let $I=I_{n,d}$ be the squarefree Veronese ideal of degree $d$. Then $d(I,k)=dk$ for all $k\ge1$.
	\end{Corollary}
	
	The following theorem plays a crucial role for our aim.\medskip
	
	\begin{Theorem}\label{Veronese:x-condition}
		Let $I=I_{n,d}$ be the squarefree Veronese ideal of degree $d$. Then $\mathcal{R}_s(I)$ satisfies the $x$-condition.
	\end{Theorem}
	\begin{proof}
		We show that $\mathcal{R}_s(I)$ satisfies the $x$-condition with respect to the order $>$ introduced in (\ref{orderT}). To this aim, let $fy_{i_1}\cdots y_{i_r}\in\ini_{<}(J)$ be a minimal monomial generator with $i_1\le\dots\le i_r$ and $f\in S$ a monomial. We need to show that $\deg(f)\leq 1$. We have a binomial relation of the form 
		$$
		b=fy_{i_1}\cdots y_{i_r}-gy_{j_1}\cdots y_{j_s}\in J
		$$
		with $g\in S$ a monomial, $j_1\le\dots\le j_s$ and $\ini_{<}(b)=fy_{i_1}\cdots y_{i_r}$. By the minimality of $\ini_{<}(b)$ we have $\gcd(f,g)=1$ and $\{i_1,\dots,i_r\}\cap\{j_1,\dots,j_s\}=\emptyset$. If $f=1$ there is nothing to prove. Suppose that $\deg(f)\ge1$. We prove that $f$ is a variable and $r\leq 2$. Let $\varphi(y_{i_\ell})=u_\ell t^{q_\ell}$ for all $\ell=1,\dots,r$ and $\varphi(y_{j_\ell})=v_\ell t^{q_\ell'}$ for all $\ell=1,\dots,s$. By the definition of the order $>$ given in (\ref{orderT}) we have $r\ge s$. Now, we distinguish the two possible cases.\smallskip
		
		\textbf{Case 1.} Let $r>s$. Since $\deg(f)\ge1$, there exists a variable $x_p\mid f$. We claim that there exist distinct integers $\ell,h$ such that $x_p\nmid u_{\ell}$ and $x_p\nmid u_h$. Indeed, if this is not the case then $\deg_{x_p}(fu_1\cdots u_r)\ge r$. Since $b$ is a relation we have $fu_1\cdots u_r=gv_1\cdots v_s$. From $\gcd(f,g)=1$ we conclude that $\deg_{x_p}(gv_1\cdots v_s)\le s<r$ which is a contradiction. Hence there exist integers $\ell<h$ such that $x_p\nmid u_{\ell}$ and $x_p\nmid u_h$. Then $y_{i_\ell}\ge y_{i_h}$ and hence by (\ref{orderY}), $q_\ell\ge q_h$. Thus
		\begin{equation}\label{eq:plh}
			x_p(u_\ell t^{q_\ell})(u_h t^{q_h})=x_c(x_pu_\ell t^{q_\ell+1})((u_h/x_c)t^{q_h-1}),
		\end{equation}
		where $x_c$ is some variable dividing $u_h$. By (\ref{eq:Rees}), the monomials $x_pu_\ell t^{q_\ell+1}$ and $(u_h/x_c)t^{q_h-1}$ are among the generators of $\mathcal{R}_s(I)$. The previous equality implies that
		\begin{equation}\label{eq:b}
			b'=x_p y_{i_\ell}y_{i_h}-x_cz y_{k}\in J 
		\end{equation}
		with $\ini_{<}(b')=x_p y_{i_\ell}y_{i_h}\in\ini_<(J)$, $\varphi(y_k)=x_pu_\ell t^{q_\ell+1}$ and $z=u_h/x_c\in S$ if $q_h=1$, or else $z=y_{k'}$ with $\varphi(y_{k'})=(u_h/x_c)t^{q_h-1}$ if $q_h>1$. Since $\ini_<(b')\mid fy_{i_1}\cdots y_{i_r}$ and $fy_{i_1}\cdots y_{i_r}\in\mathcal{G}(\ini_<(J))$ we obtain that $fy_{i_1}\cdots y_{i_r}=\ini_<(b')=x_py_{i_\ell}y_{i_h}$, as desired.\smallskip
		
		\textbf{Case 2.} Let $r=s$. Firstly, assume that there exist integers $\ell<h$ such that $u_h\nmid u_\ell$. Then, we can pick a variable $x_c$ with $x_c\mid u_h$ and $x_c\nmid u_\ell$.  Then
		$$
		(u_\ell t^{q_\ell})(u_h t^{q_h})=((x_cu_\ell )t^{q_\ell+1})((u_h/x_c)t^{q_h-1}).
		$$
		Similar to the previous argument, this equality shows that $y_{i_\ell}y_{i_h}\in\ini_<(J)$. Therefore $y_{i_\ell}y_{i_h}=fy_{i_1}\cdots y_{i_r}$, but this is not possible because $\deg(f)\ge1$. Hence, we deduce that $u_{i+1}\mid u_i$ for all $i=1,\dots,r-1$.  
		
		Now, suppose there exists $x_p\mid f$ such that $x_p$ does not divide at least two monomials among $u_1,\dots,u_r$, say $u_\ell,u_h$ with $\ell<h$. Using (\ref{eq:plh}) we obtain a relation $b'$ of the form given in (\ref{eq:b}). A similar argument as in the Case 1 shows that $fy_{i_1}\cdots y_{i_r}=x_p y_{i_\ell}y_{i_h}$, as desired.
		
		So we may assume that for any $x_p$ with $x_p\mid f$, $x_p$ divides at least $r-1$ monomial among $u_1,\dots,u_r$. Hence $\deg_{x_p}(fu_1\cdots u_r)\ge r$. From $\gcd(f,g)=1$ we have $\deg_{x_p}(gv_1\cdots v_r)\le r$. Since $fu_1\cdots u_r=gv_1\cdots v_r$ we have $\deg_{x_p}(fu_1\cdots u_r)=r$. Hence,  we obtain that $f$ is squarefree. Moreover, using that $u_{i+1}\mid u_i$ for all $i=1,\dots,r-1$ we obtain that $x_p\nmid u_r$ and $x_p\mid u_i$ for all $i=1,\dots,r-1$ and all $x_p\mid f$. Since $f$ is squarefree, we get $f\mid u_i$ for all $i=1,\dots,r-1$ and $\gcd(f,u_r)=1$. Now, we distinguish the two possible cases.\smallskip
		
		\textbf{Subcase 2.1.} Suppose there exist $p$ and $j$ with $p>j$ such that $x_p\mid f$ and $x_j\mid u_r$. Then
		$$
		x_p(u_rt^{q_r})=x_j((x_pu_r/x_j)t^{q_r}).
		$$
		Let $\varphi(y_k)=(x_pu_r/x_j)t^{q_r}$ for some $k$. Then $b'=x_py_{i_r}-x_jy_k\in J$ and $\ini_<(b')=x_py_{i_r}$ because $u_r>_{\lex}x_pu_r/x_j$. Then $fy_{i_1}\cdots y_{i_r}=x_py_{i_r}$, as desired.\smallskip
		
		\textbf{Subcase 2.2.} Suppose that the Subcase 2.1 does not hold. Since $\gcd(f,u_r)=1$ we have $\max\supp(f)<\min\supp(u_r)$. Notice that $f^r\mid fu_1\cdots u_r$ because $f\mid u_i$ for $i=1,\dots,r-1$. Therefore $f^r\mid gv_1\cdots v_r$. Since $\gcd(f,g)=1$ and $v_i$'s are squarefree, we conclude that $f\mid v_j$ for $j=1,\dots,r$. From the definition of $<$, since $\ini_<(b)=fu_1\cdots u_r$ we have either $q_r>q_r'$ or $q_r=q_r'$ and $u_r>_{\lex}v_r$. This implies that $\deg(u_r)=q_r+(d-1)\ge q'_r+(d-1)=\deg(v_r)$. Since $u_r\ne v_r$ we may choose a variable $x_c$ such that $x_c\mid u_r$ and $x_c\nmid v_r$. It follows that $x_c^r\mid fu_1\cdots u_r=gv_1\cdots v_r$. Hence $x_c\mid g$. Now choose $x_p\mid f$. Since $f\mid v_r$ we have $x_p\mid v_r$ and $p<c$ because $\max\supp(f)<\min\supp(u_r)$. Then
		$$
		x_c(v_rt^{q_r'})=x_p((x_cv_r/x_p)t^{q_r'}).
		$$
		Hence $b'=x_cy_{j_r}-x_py_k\in J$, where $\varphi(y_k)=(x_cv_r/x_p)t^{q_r'}$ and $\ini(b')=x_cy_{j_r}$ because $v_r>_{\lex}x_cv_r/x_p$. Multiplying $b'$ by $(g/x_c)y_{j_1}\cdots y_{j_{r-1}}$ we get the relation $$b''=gy_{j_1}\cdots y_{j_r}-x_p(g/x_c)y_{j_1}\cdots y_{j_{r-1}}y_k\in J.$$ Therefore
		$$
		b+b''=fy_{i_1}\cdots y_{i_r}-x_p(g/x_c)y_{j_1}y_{j_1}\cdots y_{j_{r-1}}y_k\in J
		$$
		and dividing by the common factor $x_p$ we obtain the relation
		$$
		b'''=(f/x_p)y_{i_1}\cdots y_{i_r}-(g/x_c)y_{j_1}y_{j_1}\cdots y_{j_{r-1}}y_k\in J
		$$
		with $\ini_<(b''')=(f/x_p)y_{i_1}\cdots y_{i_r}\in\ini_<(J)$. This contradicts the fact that $fy_{i_1}\cdots y_{i_r}$ is a minimal generator of $\ini_<(J)$. So this case does not happen.
	\end{proof}
	
	As a consequence of Theorem \ref{x-condition} and Theorem \ref{Veronese:x-condition} we obtain the following result which confirms Conjectures A and B for squarefree Veronese ideals.
	
	\begin{Corollary}\label{Cor:Veronese}
		Let $I=I_{n,d}$. Then 
		\begin{enumerate}
			\item [(a)] $I^{(k)}$ has linear quotients, and so it is componentwise linear, for all $k\ge1$.
			\item [(b)] $\reg\,I^{(k)}=\reg\,I^k=dk$ for all $k\ge1$.
		\end{enumerate}
	\end{Corollary}
	\begin{proof}
		(a) Equation (\ref{eq:Rees}) shows that $I$ has linear cover function $\nu_I(k)=k+(d-1)$. Hence, the result follows from Theorem \ref{x-condition} together with Theorem \ref{Veronese:x-condition}.\smallskip
		
		(b) Part (a) together with the subsequent Theorem \ref{gendeg}(c) imply that $\reg\,I^{(k)}=\reg\,I^k=dk$ for all $k\ge1$.
	\end{proof}
	
	The next result presents the generating degrees of each symbolic power of a squarefree Veronese ideal.
	
	\begin{Theorem}\label{gendeg}
		Let $I=I_{n,d}$. Then, for all $k\ge1$
		\begin{enumerate}
			\item[\textup{(a)}] We have
			$$
			I^{(k)}\ =\ \sum_{m=1}^k(\sum_{\substack{s_1+\dots+s_m=k-m\\ 0\le s_1,\dots,s_m\le n-d}}I_{n,d+s_1}\cdots I_{n,d+s_m}).
			$$
			\item[\textup{(b)}] $\beta_{0,\ell}(I^{(k)})\neq 0$ if and only if $\ell=m(d-1)+k$, where $\lceil k/(n-d+1)\rceil\le m\le k$.\smallskip
			\item[\textup{(c)}] $\omega(I^{(k)})=dk$.
		\end{enumerate}
	\end{Theorem}
	\begin{proof}
		(a) From equation~(\ref{eq:Rees}) it can be seen that $$I^{(k)}t^k\ =\ \sum I_{n,d+s_1}\cdots I_{n,d+s_m}t^{(s_1+1)+\cdots +(s_m+1)},$$ 
		where the sum is taken over all integers $1\leq m\leq k$ and all integers $s_i\geq 0$ such that $I_{n,d+s_i}\neq (0)$ and $(s_1+1)+\cdots +(s_m+1)=k$. Thus $d+s_i\leq n$ for all $i$ and $s_1+\dots+s_m=k-m$. 
		
		(b) Suppose that $\beta_{0,\ell}(I^{(k)})\neq 0$. From (a) we know that $\ell$ is the degree of a monomial in $I_{n,d+s_1}\cdots I_{n,d+s_m}$, for some integers $1\leq m\leq k$ and $0\leq s_1,\ldots,s_m\leq n-d$ with $s_1+\dots+s_m=k-m$. Therefore, $\ell=\sum_{i=1}^m (d+s_i)=md+(k-m)=m(d-1)+k$. Since $s_i\leq n-d$ for all $i$, we obtain $k-m=\sum_{i=1}^m s_i\leq m(n-d)$, which implies that $m\geq \lceil k/(n-d+1)\rceil$.
		
		Conversely, let $\lceil k/(n-d+1)\rceil\le m\le k$ be an integer and let $\ell=m(d-1)+k$.
		The inequality $k/(n-d+1)\leq m$ implies that $k-m\leq m(n-d)$. Therefore, there exist integers $s_1,\ldots,s_m$ such that $0\leq s_i\leq n-d$ for all $i$ and $s_1+\dots+s_m=k-m$. For such integers $s_i$ we have $I_{n,d+s_i}\neq (0)$ and $(0)\neq I_{n,d+s_1}\cdots I_{n,d+s_m}\subset I^{(k)}$. We may assume that $s_1\le\cdots\le s_m$. Set $u_{\ell}=\prod_{i=1}^{d+s_{\ell}}x_i$ for $1\leq \ell\leq m$. Then we have $u=u_1u_2\cdots u_m\in I_{n,d+s_1}\cdots I_{n,d+s_m}\subset I^{(k)}$ with 
		$\deg(u)=md+\sum_{i=1}^m s_i=md+(k-m)=m(d-1)+k=\ell$.
		Hence, in order to show that $\beta_{0,\ell}(I^{(k)})\neq 0$, it is enough to show that $u$ is a minimal generator of $I^{(k)}$.
		Suppose by contradiction that this is not the case, and a monomial $v\in I^{(k)}$ strictly divides $u$. From (a) we know that $v\in I_{n,d+s'_1}\cdots I_{n,d+s'_p}$ for integers $1\le p\le k$ and $s'_i\ge 0$ with   
		$s'_1+\dots+s'_p=k-p$. Moreover, $p(d-1)+k=\deg(v)<\deg(u)=m(d-1)+k$. Therefore, $p<m$. 
		Let $u=x_1^{a_1}\cdots x_n^{a_n}$ and $v=x_1^{b_1}\cdots x_n^{b_n}$. We have $b_i\le a_i$  and $b_i\le p$ for all $i$, since $v$ is the product of $p$ squarefree monomials. So
		\begin{equation}\label{eq:sump}
			\begin{aligned}
				p(d-1)+k\ =\ \deg(v)\ &= \sum_{i=1}^{d+s_{m-p+1}}b_i+\sum_{i=d+s_{m-p+1}+1}^{n}b_i \\
				& \le\ p(d+s_{m-p+1})+\sum_{i=d+s_{m-p+1}+1}^{n}b_i.
			\end{aligned}
		\end{equation}
		Notice that 
		\begin{equation*}\label{eq:u}
			u=(\prod_{i=1}^{d+s_1}x_i^m)(\prod_{i=d+s_1+1}^{d+s_2}x_i^{m-1})\cdots(\prod_{i=d+s_{m-2}+1}^{d+s_{m-1}}x_i^2)(\prod_{i=d+s_{m-1}+1}^{d+s_m}x_i).
		\end{equation*}
		Hence, $a_i=0$ for $i>d+s_m$ and $a_i=q$, for any $d+s_{m-q}+1\leq i\leq d+s_{m-q+1}$, where $1\leq q\leq m$ and $s_0=-d$. Therefore,
		\begin{equation}\label{eq:p2}
			\begin{aligned}
				\sum_{i=d+s_{m-p+1}+1}^{n}b_i\ &\le\ \sum_{i=d+s_{m-p+1}+1}^{n}a_i\\
				&\le\ (p-1)(s_{m-p+2}-s_{m-p+1})+\cdots+3(s_{m-2}-s_{m-3})\\
				&\phantom{aaaa}+2(s_{m-1}-s_{m-2})+(s_m-s_{m-1})\\
				&=\ \sum_{\ell=1}^{p-1}\ell(s_{m-\ell+1}-s_{m-\ell})=\sum_{\ell=0}^{p-2}s_{m-\ell}-(p-1)s_{m-p+1}. 
			\end{aligned}
		\end{equation} 
		From equations (\ref{eq:sump}) and (\ref{eq:p2}) we obtain
		$$p(d-1)+k\le p(d+s_{m-p+1})+\sum_{\ell=0}^{p-2}s_{m-\ell}-(p-1)s_{m-p+1}=pd+\sum_{\ell=0}^{p-1}s_{m-\ell}.$$
		This implies that $k-p\le \sum_{\ell=0}^{p-1}s_{m-\ell}\le s_1+\cdots+s_m=k-m$. Hence, we get $p\ge m$, which is a contradiction. 
		
		(c) follows from (b).
	\end{proof}
	
	As a corollary of Theorem~\ref{gendeg} we obtain the least degree of a generator of $I_{n,d}^{(k)}$ and recover \cite[Theorem 7.5]{BC} on the Waldschmidt constant of squarefree Veronese ideals. Recall that for a homogeneous ideal $I\subset S$,  the {\em Waldschmidt constant} of $I$ is defined as $\widehat{\alpha}(I)=\lim_{k\to \infty} \alpha(I^{(k)})/k$. 
	
	\begin{Corollary}
		Let $I=I_{n,d}$. Then $\alpha(I^{(k)})=(d-1)\lceil k/(n-d+1)\rceil+k$. In particular,  $\widehat{\alpha}(I)=n/(n-d+1)$.    
	\end{Corollary}
	
	\label{notcp}Notice that in general $I_{n,d}^{(k)}$  is not componentwise polymatroidal. Indeed, consider the ideal $I_{4,3}=(x_1x_2x_3,\:x_1x_2x_4,\:x_1x_3x_4,\:x_2x_3x_4)$. Then,
	$$
	I_{4,3}^{(2)}=(x_1x_2x_3x_4,\:x_1^2x_2^2x_3^2,\:x_1^2x_2^2x_4^2,\:x_1^2x_3^2x_4^2,\:x_2^2x_3^2x_4^2)
	$$
	is not componentwise polymatroidal. In fact, we have
	\begin{align*}
		L=(I_{4,3}^{(2)})_{\langle6\rangle}\ = \ (&x_{1}^{3}x_{2}x_{3}x_{4},\:x_{1}^{2}x_{2}^{2}x_{3}^{2},\:x_{1}^{2}x_{2}^{2}x_{3}x_{4},\:x_{1}^{2}x_{2}^{2}x_{4}^{2},\:x_{1}^{2}x_{2}x_{3}^{2}
		x_{4},\:x_{1}^{2}x_{2}x_{3}x_{4}^{2},\\
		&\:x_{1}^{2}x_{3}^{2}x_{4}^{2},\:x_{1}x_{2}^{3}x_{3}x_{4},\:x_{1}x_{2}^{2}x_{3}^{2}x_{4},\:x_{1}x_{2}^{2}x_{3}x_{
			4}^{2},\:x_{1}x_{2}x_{3}^{3}x_{4},\\&\:x_{1}x_{2}x_{3}^{2}x_{4}^{2},\:x_{1}x_{2}x_{3}x_{4}^{3},\:x_{2}^{2}x_{3}^{2}x_{4}^{2}),
	\end{align*}
	where the minimal monomial generators of $L$ are ordered according to the lex order $>_{\lex}$ induced by $x_1>x_2>x_3>x_4$. Notice that $L$ does not have linear quotients with respect to such order. Indeed, $(x_{1}^{3}x_{2}x_{3}x_{4}):(x_{1}^{2}x_{2}^{2}x_{3}^{2})=(x_1x_4)$ is not generated by variables. Hence $L$ is not polymatroidal by \cite[Theorem 2.4]{BR}.\smallskip
	
	The next result shows that $I_{2,n}^{(k)}$ is a componentwise polymatroidal ideal. Notice that $I_{2,d}$ may be viewed as an edge ideal.
	
	Let $G$ be a finite simple graph on the vertex set $V(G)=\{x_1,\dots,x_n\}$ with the edge set $E(G)$. The \textit{complementary graph} of $G$ is the graph $G^c$ with the vertex set $V(G^c)=V(G)$ and the edge set $E(G^c)=\{\{x_i,x_j\}:\ i\ne j,\,\{x_i,x_j\}\notin E(G)\}$.
	
	\begin{Proposition}\label{Prop:ndcp}
		Let $I=I_{n,d}$. If $d\in\{1,2,n\}$, then $I^{(k)}$ is componentwise polymatroidal for all $k\ge1$.
	\end{Proposition}
	\begin{proof}
		For $d=1$ we have $I=\m$ and so $I^{(k)}=I^k=\m^k$ is (componentwise) polymatroidal for all $k\ge1$. Similarly, for $d=n$ we have $I=(x_1\cdots x_n)$ and so $I^{(k)}=I^k=(x_1^kx_2^k\cdots x_n^k)$ is (componentwise) polymatroidal for all $k\ge1$. Finally, for $d=2$ we have $I=I(G)$, where $G$ is the complete graph on the vertex set $\{x_1,\dots,x_n\}$. Notice that $G^c$ is the graph on the vertex set $\{x_1,\dots,x_n\}$ whose all vertices are isolated. Hence $G^c$ is a block graph and any ordering of the variables $x_1,\dots,x_n$ is a perfect elimination order of $G^c$ in the sense of Dirac \cite{Dirac}. Then, the proof of \cite[Theorem 2.3(a)]{FMR} shows that $(I^{(k)})_{\langle j\rangle}$ has linear quotients with respect to the lex order induced by any ordering of the variables $x_1,\dots,x_n$, for any $j$. Hence, by a result of Bandari and Rahmati-Asghar \cite[Theorem 2.4]{BR} it follows that $(I^{(k)})_{\langle j\rangle}$ is polymatroidal, for all $j$ and $k$. Hence, $I^{(k)}$ is componentwise polymatroidal. Alternatively, notice that $I_{n,2}^{(k)}=\bigcap_{i=1}^nP_i^k$ for all $k\ge1$, where $P_i=P_{[n]\setminus\{i\}}$. Since $P_i+P_j=\m$ for $i\ne j$, by a result of Francisco and Van Tuyl \cite[Theorem 3.1]{FVT} it follows that $I_{n,2}^{(k)}$ is componentwise polymatroidal for all $k\ge1$.
	\end{proof}
	
	We expect that the converse of Proposition \ref{Prop:ndcp} holds as well.\smallskip
	
	As was shown in the proof of Theorem~\ref{Veronese:x-condition}, the initial ideal of the defining ideal $J$ of $\mathcal{R}_s(I_{n,d})$ is generated by monomials of the form $v$ and $x_{\ell}v'$ with $v,v'\in K[y_1,\ldots,y_m]$ such that $\deg(v')\leq 2$. In general $\ini_<(J)$ may not be a quadratic monomial ideal. For instance, consider $I=I_{4,3}$. Let $T=K[x_1,\dots,x_4,y_1,\dots,y_5]$ and let $\varphi:T\rightarrow\mathcal{R}_s(I)$ be the $S$-algebra map defined by setting $\varphi(y_1)=x_1x_2x_3x_4t^2$, $\varphi(y_2)=x_1x_2x_3t$, $\varphi(y_3)=x_1x_2x_4t$, $\varphi(y_4)=x_1x_3x_4t$ and $\varphi(y_5)=x_2x_3x_4t$. \textit{Macaulay2} \cite{GDS} shows that
	$$
	\ini_<(\Ker\,\varphi)=(x_{1}y_{5}^{2},\,x_{2}y_{4},\,x_{3}y_{3},\,x_{4}y_{2},\,y_{2}y_{3},\,y_{2}y_{4},\,y_{2}y_{5},\,y_{3}y_{4},\,y_{3}y_{5},\,y_{4}y_{5})
	$$
	with respect to the order (\ref{orderT}).
	Nonetheless, in the next proposition we show that $\ini_<(J)$ is generated in degree at most three.  
	
	\begin{Proposition}\label{Prop:ini3}
		Let $I=I_{n,d}$, and let $J$ be the defining ideal of $\mathcal{R}_s(I)$. Then the initial ideal $\ini_<(J)$ with respect to the order \textup{(\ref{orderT})} is generated by monomials of the form $x_iy_j$, $y_jy_k$, and $x_iy_jy_k$.
	\end{Proposition}
	\begin{proof}
		We keep the notation as in the proof of Theorem~\ref{Veronese:x-condition}. Let $fy_{i_1}\cdots y_{i_r}\in\ini_{<}(J)$ be a minimal monomial generator with $i_1\le\dots\le i_r$ and $f\in S$ a monomial. 
		If $f\neq 1$, then the proof of Theorem~\ref{Veronese:x-condition} shows that either $fy_{i_1}\cdots y_{i_r}=x_iy_j$ for some $i$ and $j$ or $fy_{i_1}\cdots y_{i_r}=x_iy_jy_k$ for some $i,j,k$. Now, suppose that $f=1$.
		Then we have a binomial relation of the form
		$$
		b=y_{i_1}\cdots y_{i_r}-gy_{j_1}\cdots y_{j_s}\in J
		$$
		with $g\in S$ a monomial, $j_1\le\dots\le j_s$ and $\ini_{<}(b)=y_{i_1}\cdots y_{i_r}$. By the minimality of $\ini_{<}(b)$ we have $\{i_1,\dots,i_r\}\cap\{j_1,\dots,j_s\}=\emptyset$. We show that $\deg(y_{i_1}\cdots y_{i_r})=2$. 
		
		Let $\varphi(y_{i_\ell})=u_\ell t^{q_\ell}$ for all $\ell=1,\dots,r$ and $\varphi(y_{j_\ell})=v_\ell t^{q_\ell'}$ for all $\ell=1,\dots,s$.
		
		\textbf{Case 1.} Assume that there exist integers $\ell<h$ such that $u_h\nmid u_\ell$. Then, we can pick a variable $x_c$ with $x_c\mid u_h$ and $x_c\nmid u_\ell$.  Then
		$$
		(u_\ell t^{q_\ell})(u_h t^{q_h})=((x_cu_\ell )t^{q_\ell+1})((u_h/x_c)t^{q_h-1}).
		$$
		This equality shows that $y_{i_\ell}y_{i_h}\in\ini_<(J)$. Since $y_{i_\ell}y_{i_h}$ divides $y_{i_1}\cdots y_{i_r}$ and the latter one is a minimal generator of $\ini_<(J)$, we get $y_{i_1}\cdots y_{i_r}=y_{i_\ell}y_{i_h}$, as desired.
		
		\textbf{Case 2.} Suppose that Case 1 does not hold. Then $u_{i+1}\mid u_i$ for all $i=1,\dots,r-1$. We show that this case does not happen. 
		First we show that under this assumption $g=1$. By equation (\ref{eq:Rees}) we have $\deg(u_\ell)=q_\ell+d-1$ for all $1\le \ell\le r$ and 
		$\deg(v_\ell)=q'_\ell+d-1$ for all $1\le \ell\le s$. Moreover, $u_1\cdots u_rt^{q_1+\cdots+q_r}=gv_1\cdots v_st^{q'_1+\cdots+q'_s}$. Hence, $$\sum_{\ell=1}^r (q_\ell+d-1)=\deg(u_1\cdots u_r)=\deg(g)+\deg(v_1\cdots v_s)=\deg(g)+\sum_{\ell=1}^s (q'_\ell+d-1).$$
		From this and the equality $\sum_{\ell=1}^r q_\ell=\sum_{\ell=1}^s q'_\ell$, we obtain $r(d-1)=\deg(g)+s(d-1)$. Notice that since $\ini_{<}(b)=y_{i_1}\cdots y_{i_r}$, we have $r\ge s$. Then $\deg(g)=(r-s)(d-1)$.   
		Since $u_r\mid u_i$ for all $i$, we have $u_r^r\mid u_1\cdots u_r=gv_1\cdots v_s$. From the fact that all $v_i$'s are squarefree, we obtain $u_r^{r-s}\mid g$. Thus $$(r-s)(q_r+d-1)=\deg(u_r^{r-s})\le \deg(g)=(r-s)(d-1).$$ This would imply that $(r-s)q_r\le 0$. Since $q_r>0$ and $r\ge s$, we obtain $r=s$. Thus $\deg(g)=(r-s)(d-1)=0$, which means $g=1$ and $b=y_{i_1}\cdots y_{i_r}-y_{j_1}\cdots y_{j_r}$. Moreover, $u_r^r\mid v_1\cdots v_r$, which implies that $u_r\mid v_i$ for all $i$. In particular, $u_r\mid v_r$. So we conclude that $v_r>_{\lex} u_r$ and $q'_r+(d-1)=\deg(v_r)\ge \deg(u_r)=q_r+(d-1)$. Thus $q'_r\ge q_r$ and by (\ref{orderY}) we have $y_{j_r}>y_{i_r}$. Hence, 
		$y_{j_1}\cdots y_{j_r}>_{\rlex} y_{i_1}\cdots y_{i_r}$. This means that $\ini_{<}(b)=y_{j_1}\cdots y_{j_r}$, a contradiction.
	\end{proof}
	
	We close this section with the following question.
	\begin{Question}\label{Q:quadratic}
		When does the defining ideal of $\mathcal{R}_s(I_{n,d})$ have a quadratic Gr\"obner basis for some monomial order $<$ ?
	\end{Question} 
	
	\section{Matching-matroidal ideals}\label{sec4}
	
	In this section, we consider matroidal ideals attached to matching matroids and we prove Conjectures \ref{ConjA} and \ref{ConjB} for some special families of such ideals.
	
	Let $\mathcal{A}=\{A_1,\dots,A_t\}$ be a finite collection of non-empty subsets of $[n]$. A \textit{transversal} of $\mathcal{A}$ is the image of an injective map $\psi:[t]\rightarrow[n]$ with the property that $\psi(j)\in A_j$ for all $j=1,\dots,t$. It is well-known that the collection of all transversals of $\mathcal{A}$ is the set of bases of a matroid, called a \textit{transversal matroid}.
	Next, we provide a short algebraic proof of this fact.

	Given ${\bf b}=(b_1,\dots,b_n)\in\ZZ_{\ge0}^n$ we put ${\bf x^b}=x_1^{b_1}\cdots x_n^{b_n}$. For a monomial ideal $I\subset S$ and a vector ${\bf a}=(a_1,\dots,a_n)\in\ZZ_{\ge0}^n$, we define the ideal $I^{\le{\bf a}}$ as the monomial ideal generated by all monomials ${\bf x^b}\in I$ such that ${\bf b}\le{\bf a}$, that is $b_i\le a_i$ for all $i$.
	
	\begin{Lemma}\label{Lem:<a}
		Let $I\subset S$ be a polymatroidal ideal. Then $I^{\le\bf a}$ is polymatroidal too.
	\end{Lemma}
	\begin{proof}
		Let $u,v\in\mathcal{G}(I^{\le{\bf a}})$ with $\deg_{x_i}(u)>\deg_{x_i}(v)$. Since $u,v\in\mathcal{G}(I)$ and $I$ is polymatroidal, there exists $j$ with $\deg_{x_j}(u)<\deg_{x_j}(v)$ such that $x_j(u/x_i)\in\mathcal{G}(I)$. We claim that $x_j(u/x_i)\in\mathcal{G}(I^{\le{\bf a}})$. Indeed, $\deg_{x_s}(x_j(u/x_i))\le\deg_{x_s}(u)\le a_s$ for all integers $s\ne j$ and $\deg_{x_j}(x_j(u/x_i))=\deg_{x_j}(u)+1\le\deg_{x_j}(v)\le a_j$. This shows that $x_j(u/x_i)\in\mathcal{G}(I^{\le{\bf a}})$. Hence $I^{\le{\bf a}}$ is polymatroidal.
	\end{proof}
	
	As a consequence of Lemma \ref{Lem:<a} we have
	
	\begin{Corollary}\label{Cor:mat-mat}
		Let $\mathcal{A}=\{A_1,\dots,A_t\}$ be a collection of non-empty subsets of $[n]$. The collection of transversals of $\mathcal{A}$ is the set of bases of a matroid.
	\end{Corollary}
	\begin{proof}
		Let $J=P_{A_1}\cdots P_{A_t}$. Since each $P_{A_i}$ is polymatroidal, $J$ is polymatroidal too. Let $I=J^{\le{\bf 1}}$ where ${\bf 1}=(1,1,\dots,1)$. By Lemma \ref{Lem:<a} it follows that $I$ is a matroidal ideal. Let $B$ be the set of the transversals of $\mathcal{A}$. We claim that the set consisting of the supports of the minimal monomial generators of $I$ coincides with $B$. This will imply the assertion.
		
		A minimal monomial generator of $J$ is of the form $u=x_{i_1}\cdots x_{i_t}$ with $i_j\in A_j$ for all $j=1,\dots,t$. This monomial $u$ belongs to $\mathcal{G}(I)$ if and only if $u$ is squarefree. That is, if and only if $i_j\ne i_\ell$ for all $1\le i<j\le t$. So the map $\psi_u:[t]\rightarrow[n]$ defined by $\psi_u(j)=i_j$ is injective and $\psi_u(j)\in A_j$ for all $j$. Hence $\Im\,\psi=\supp(u)\in B$. Conversely, let $\{i_1',\dots,i_t'\}\in B$ be a transversal. Then $x_{i_j'}\in P_{A_j}$ for all $j$ and the monomial $u'=x_{i_1'}\cdots x_{i_t'}$ is squarefree. This shows that $u'\in\mathcal{G}(I)$, as desired.
	\end{proof}
	
	Next, we provide a different description, due to Edmonds and Fulkerson, of the matroids considered in this section.
	
	Let $G$ be a finite simple graph on the vertex set $[n]$ with the edge set $E(G)$. A {\em $k$-matching} of $G$ is a collection of $k$ edges of $G$ which are pairwise disjoint. The \textit{matching number} of $G$, denoted by $\nu(G)$, is the maximum size of a matching of $G$.

	We say that a subset $B\subseteq [n]$ {\em meets} a matching $M$, if $B\subseteq \bigcup_{e\in M} e$. Now, let $A\subseteq[n]$ be  non-empty. Then
	$$
	\{B:\ B\subseteq A,\ B \textup{ meets}\ M \textup{ for some}\ \textup{matching}\ M\ \textup{of}\ G\}
	$$
	is the collection of independent sets of a matroid. We denote this matroid by $\Match(G;A)$ and call it a \textit{matching-matroid}. This fact was proved by Edmonds and Fulkerson (see \cite[Theorem 1 on page 246]{W}).
	
	When $A=[n]$, we denote $\Match(G;[n])$ simply by $\Match(G)$. An algebraic proof that $\Match(G)$ is a matroid can be found for instance in \cite[Theorem 1.1]{EFnote}. See also \cite[Theorem 1.7]{EF}, \cite[Theorem 1.8]{FM} and \cite[Theorem 4.3]{HF} for some related algebraic results.
	
	We say that a squarefree monomial ideal $I$ is \textit{matching-matroidal} if the supports of the minimal monomial generators of $I$ are the bases of a matching-matroid.
	
	A deep theorem of Edmonds and Fulkerson shows that a matroid is transversal if and only if it is a matching-matroid (see \cite[Theorem 2 on page 248]{W}). Let ${\bf 1}=(1,1,\dots,1)$. For a monomial ideal $I\subset S$, we call $I^{\le{\bf1}}$ the \textit{squarefree part} of $I$.
	
	Using the theorem of Edmonds and Fulkerson combined with Corollary \ref{Cor:mat-mat} we obtain the following algebraic characterization. 
	\begin{Theorem}\label{Thm:mm}
		Let $I\subset S$ be a monomial ideal. Then $I$ is matching-matroidal if and only if $I=(P_1\cdots P_t)^{\le{\bf 1}}$ for some monomial prime ideals $P_1,\dots,P_t$.
	\end{Theorem}
	
	At the moment we are not able to prove Conjectures \ref{ConjA} and \ref{ConjB} for all matching-matroidal ideals. Therefore, in what follows we restrict our attention to a special but wide family of matching-matroidal ideals.

	Given a non-empty subset $A$ of $[n]$ we denote by $I_{A,d}$ the squarefree Veronese ideal of degree $d$ in the polynomial ring $S_A=K[x_i:\ i\in A]$.
	
	Let $\mathcal{A}=\{A_1,\dots,A_t\}$ be a finite collection of non-empty subsets of $[n]$. We say that $\mathcal{A}$ is of \textit{Veronese type}, if whenever we have $A_i\cap A_j\ne\emptyset$ then $A_i=A_j$. Let $I$ be the matching-matroidal ideal attached to a set $\mathcal{A}$ of Veronese type, that is $I=(P_{A_1}\cdots P_{A_t})^{\le{\bf1}}$. We say that $I$ is a \textit{matching-matroidal ideal of Veronese type}. Notice that for such an ideal $I$ we have $I=(P_{B_1}^{d_1}\cdots P_{B_m}^{d_m})^{\le{\bf 1}}$, where $B_i\cap B_j=\emptyset$ for all $i\ne j$. Since the ideals $P_{B_i}^{d_i}$ have pairwise disjoint supports and $(P_{B_i}^{d_i})^{\le{\bf 1}}=I_{B_i,d_i}$ we obtain that
	$$
	I=(P_{B_1}^{d_1}\cdots P_{B_m}^{d_m})^{\le{\bf 1}}=(P_{B_1}^{d_1})^{\le{\bf 1}}\cdots(P_{B_m}^{d_m})^{\le{\bf 1}}=I_{B_1,d_1}\cdots I_{B_m,d_m}.
	$$
	
	Hence, we have shown
	\begin{Proposition}\label{Prop:mat-mat}
		A monomial ideal $I$ is a matching-matroidal ideal of Veronese type if and only if $I$ is the product of squarefree Veronese ideals with pairwise disjoint supports.
	\end{Proposition}
	
	The following remark follows easily from the definition of symbolic powers.
	
	\begin{Remark}\label{Rem}
		Let $I_1,\dots,I_t\subset S$ be monomial ideals with pairwise disjoint supports. Then $(I_1\cdots I_t)^{(k)}=I_1^{(k)}\cdots I_t^{(k)}$ for all $k\ge1$.
	\end{Remark}

	Now, we verify Conjectures \ref{ConjA}, \ref{ConjB} for matching-matroidal ideals of Veronese type.
	\begin{Theorem}\label{Thm:mat-mat}
		Let $I\subset S$ be a matching-matroidal ideal of Veronese type. Then, $I^{(k)}$ has linear quotients and $\reg\,I^{(k)}=\reg\,I^k$ for all $k\ge1$.
	\end{Theorem}
	
	For the proof of this theorem, we need the following lemma.
	
	\begin{Lemma}\label{Lem:I-dis}
		Let $I_1,\dots,I_t\subset S$ be monomial ideals with pairwise disjoint supports. Then the following statements hold.
		\begin{enumerate}
			\item[(a)] Suppose that $I_j$ is componentwise linear for all $j=1,\dots,t$. Then $I_1\cdots I_t$ is componentwise linear.
			\item[(b)] Suppose that $I_j$ have linear quotients for all $j=1,\dots,t$. Then $I_1\cdots I_t$ has linear quotients.
			\item[(c)] Suppose that $\reg\,I_j^{(k)}=a_jk$ for all $j=1,\dots,t$ and all $k\ge1$. Then $\reg\,(I_1\cdots I_t)^{(k)}=(a_1+\dots+a_t)k$ for all $k\ge 1$.
		\end{enumerate}
	\end{Lemma}
	\begin{proof}
		To prove the statements (a), (b) and (c), proceeding by induction on $t$, it is enough to prove the case $t=2$, with the base case $t=1$ being trivial. Therefore, it is enough to consider two monomial ideals $I,J$ of $S$ with $\supp(I)\cap\supp(J)=\emptyset$.
		
		(a) It is clear that for all $d$ we have
		$$
		(IJ)_{\langle d\rangle}\ =\ \sum_{i+j=d}I_{\langle i\rangle}J_{\langle j\rangle}\ =\ \sum_{i=0}^dI_{\langle d-i\rangle}J_{\langle i\rangle}.
		$$
		We claim that for all $j=0,\dots,d$ the ideal $H_j=\sum_{i=0}^j I_{\langle d-i\rangle}J_{\langle i\rangle}$ has a $d$-linear resolution. Since $H_d=(IJ)_{\langle d\rangle}$ this will yield the assertion. To prove our claim we proceed by finite induction. For $j=0$, $H_0=I_{\langle d\rangle}$ has $d$-linear resolution by assumption. Now let $j>0$. Notice that $H_j=H_{j-1}+I_{\langle d-j\rangle}J_{\langle j\rangle}$. By induction $H_{j-1}$ has $d$-linear resolution. Furthermore $I_{\langle d-j\rangle}J_{\langle j\rangle}$ has $d$-linear resolution for it is the product of monomial ideals with disjoint supports having $(d-j)$-linear and $j$-linear resolution, respectively. It follows from \cite[Corollary 2.4]{FHT} that $H_j=H_{j-1}+I_{\langle d-j\rangle}J_{\langle j\rangle}$ is a Betti splitting. Now, using \cite[Proposition 1.8]{CFts1} we have that $H_{j}$ has $d$-linear resolution if and only if $H_{j-1}\cap I_{\langle d-j\rangle}J_{\langle j\rangle}$ has a $(d+1)$-linear resolution. Notice that $I_{\langle d\rangle}\subset I_{\langle d-1\rangle}\subset\cdots\subset I_{\langle d-j\rangle}$ and $J_{\langle 0\rangle}\supset J_{\langle 1\rangle}\supset\cdots\supset J_{\langle j\rangle}$. Taking into account that the sum of monomial ideals distributes over intersections, we obtain that
		\begin{align*}
			H_{j-1}\cap I_{\langle d-j\rangle}J_{\langle j\rangle}\ &=\ (\sum_{i=0}^{j-1}I_{\langle d-i\rangle}J_{\langle i\rangle})\cap I_{\langle d-j\rangle}J_{\langle j\rangle}\ =\ \sum_{i=0}^{j-1}[(I_{\langle d-i\rangle}J_{\langle i\rangle})\cap(I_{\langle d-j\rangle}J_{\langle j\rangle})]\\
			&=\ \phantom{,}\sum_{i=0}^{j-1}[(I_{\langle d-i\rangle}\cap I_{\langle d-j\rangle})(J_{\langle i\rangle}\cap J_{\langle j\rangle})]\ =\ \sum_{i=0}^{j-1}I_{\langle d-i\rangle}J_{\langle j\rangle}\\
			&=\ (\sum_{i=0}^{j-1} I_{\langle d-i\rangle})J_{\langle j\rangle}\ =\ I_{\langle d-j+1\rangle}J_{\langle j\rangle}
		\end{align*}
		has $(d+1)$-linear resolution. Hence $(IJ)_{\langle d\rangle}$ has $d$-linear resolution for all $d$, as desired.\smallskip
		
		(b) Let $u_1,\dots,u_m$ and $v_1,\dots,v_\ell$ be linear quotients orders of $I$ and $J$, respectively. Since $\supp(I)\cap\supp(J)=\emptyset$, we have $\mathcal{G}(IJ)=\{u_iv_j:\ i=1,\dots,m,j=1,\dots,\ell\}$. We claim that
		$$
		u_1v_1,u_2v_1,\dots,u_mv_1,\, u_1v_2,\dots,u_mv_2,\,\dots,\,u_1v_\ell,\dots,u_mv_\ell
		$$
		is a linear quotients order of $IJ$. To this end, consider two generators $u_iv_j>u_pv_q$. Then, either $j<q$ or $j=q$ and $i<p$.
		
		Suppose that $j<q$. Since $J$ has linear quotients, there exists $r<q$ such that $v_r:v_q=x_s$ divides $v_j:v_q$. Since $u_iv_j:u_pv_q=(u_i:u_p)(v_j:v_q)$ it follows that $x_s$ divides $u_iv_j:u_pv_q$ too. Consider the monomial $u_pv_r$. Notice that $u_pv_r>u_pv_q$ because $r<q$. Moreover $u_pv_r:u_pv_q=v_r:v_q=x_s$ divides $u_iv_j:u_pv_q$, as desired.
		
		Suppose now $j=q$, then $u_iv_j:u_pv_q=u_i:u_p$. Since $I$ has linear quotients, there exists $r<p$ such that $u_r:u_p=x_s$ divides $u_i:u_p$. Then $u_rv_j>u_pv_j$ and $u_rv_j:u_pv_j=x_s$ divides $u_iv_j:u_pv_q=u_i:u_p$, as desired.\smallskip
		
		(c) By Remark \ref{Rem} we have $(IJ)^{(k)}=I^{(k)}J^{(k)}$ for all $k\ge1$. By the assumption, $\reg\,I^{(k)}=ak$ and $\reg\,J^{(k)}=bk$ for all $k\ge1$. Since $\supp(I^{(k)})\cap\supp(J^{(k)})=\emptyset$, we have $\reg\,(IJ)^{(k)}=\reg\,I^{(k)}J^{(k)}=\reg\,I^{(k)}+\reg\,J^{(k)}=(a+b)k$ for all $k\ge1$.
	\end{proof}
	
	Now, we are in the position to prove Theorem \ref{Thm:mat-mat}.
	\begin{proof}[Proof of Theorem \ref{Thm:mat-mat}]
		By Proposition \ref{Prop:mat-mat}, $I=I_{B_1,d_1}\cdots I_{B_m,d_m}$ with $B_i\cap B_j=\emptyset$ for all $1\le i<j\le m$. Using Remark \ref{Rem} we have $I^{(k)}=I_{B_1,d_1}^{(k)}\cdots I_{B_m,d_m}^{(k)}$ for all $k\ge1$ with $\supp(I_{B_i,d_i}^{(k)})\cap\supp(I_{B_j,d_j}^{(k)})=\emptyset$ for all $1\le i<j\le m$. The assertion now follows by combining Lemma \ref{Lem:I-dis}(b)-(c) with Corollary \ref{Cor:Veronese}.
	\end{proof}
	
	\section{Comparison between ordinary and symbolic powers}\label{sec5}
	
	Let $I\subset S$ be a monomial ideal. Notice that $I^k\subseteq I^{(k)}$. However, in general $I^{(k)}$ is a much larger ideal compared to $I^k$. In this section, we aim to classify the polymatroidal ideals $I$ for which $I^{(k)}=I^k$ for all $k\ge1$.\smallskip
	
	Firstly, we consider the squarefree case, i.e., $I$ is matroidal.
	
	Let $I\subset S$ be a squarefree monomial ideal. We say that $I$ is \textit{K\"onig} if $I$ contains a monomial regular sequence of length $\height(I)$. We say that $I$ is \textit{packed} if every ideal obtained from $I$ by setting a (possibly empty) subset of the variables equal to $0$ and a disjoint (possibly empty) subset of the variables equal to $1$ is K\"onig.
	
	A famous conjecture posed by Conforti and Cornu\'ejols \cite{CC} has been restated equivalently in algebraic terms by Gitler, Villarreal and others \cite{GRV,GVV} as follows.
	
	\begin{Conj}\label{Conj:CC}
		A squarefree monomial ideal $I\subset S$ satisfies $I^{(k)}=I^k$ for all $k\ge1$ if and only if $I$ is packed.
	\end{Conj}
	
	This conjecture is sometimes referred as the \textit{Packing problem}. It is known that Conjecture \ref{Conj:CC} holds for any squarefree monomial ideal generated in degree two, that is for any edge ideal, see \cite[Proposition 4.27]{GRV}. Moreover, the following implication is well-known.
	
	\begin{Proposition}\label{Prop:Pack}
		Let $I\subset S$ be a squarefree monomial ideal. Suppose that $I^{(k)}=I^k$ for all $k\ge1$. Then $I$ is packed.
	\end{Proposition}
	
	Following \cite{MK} we say that a (squarefree) monomial ideal $I\subset S$ is \textit{vertex splittable} if $I$ can be obtained by the following recursive procedure.
	\begin{itemize}
		\item[(i)] If $u$ is a monomial and $I=(u)$, $I=(0)$ or $I=S$, then $I$ is vertex splittable.\smallskip
		\item[(ii)] If there is a variable $x_i$ and vertex splittable ideals $I_1$ and $I_2$ such that $I=x_iI_1+I_2$, $i\notin\supp(I_2)$, $I_2\subset I_1$ and $\mathcal{G}(I)$ is the disjoint union of $\mathcal{G}(x_iI_1)$ and $\mathcal{G}(I_2)$, then $I$ is vertex splittable.
	\end{itemize}
	
	The following result was shown in \cite[Proposition 2]{CF2024}.
	\begin{Proposition}
		$($Componentwise$)$ polymatroidal ideals are vertex splittable.
	\end{Proposition}
	
	By \cite[Lemma 2.1]{HH2003} any polymatroidal ideal satisfies the so-called \textit{dual exchange property}, namely: for all $u,v\in\mathcal{G}(I)$ and all $i$ such that $\deg_{x_i}(u)<\deg_{x_i}(v)$ there exists $j$ with $\deg_{x_j}(u)>\deg_{x_j}(v)$ such that $x_i(u/x_j)\in\mathcal{G}(I)$.\smallskip
	
	For the proof of Theorem \ref{Thm:I(k)k-sq} we need also the following splitting decomposition which was proved in \cite[Lemma 5.6]{F2023}, (see also the proof of \cite[Theorem 1.1]{BH2013}). We regard the zero ideal $(0)$ as a polymatroidal ideal.
	
	\begin{Lemma}\label{Lem:againP}
		Let $I\subset S$ be a polymatroidal ideal generated in degree $\alpha(I)=d\ge2$. For any $i\in\supp(I)$ there exist polymatroidal ideals $I_1,I_2\subset S$ with $\alpha(I_1)=d-1$ such that $I=x_iI_1+I_2$, $i\notin\supp(I_2)$ and $I_2\subset I_1$.
	\end{Lemma}
	
	For a graded ideal $I\subset S$, we denote by $\mu(I)=\dim_K(I/\m I)$ the minimal number of generators of $I$. Notice that if $I$ is a monomial ideal, then $\mu(I)=|\mathcal{G}(I)|$.
	
	Now, we establish Conjecture \ref{Conj:CC} for matroidal ideals.
	
	\begin{Theorem}\label{Thm:I(k)k-sq}
		For a matroidal ideal $I\subset S$ the following conditions are equivalent.
		\begin{enumerate}
			\item[(a)] $I=P_{F_1}\cdots P_{F_t}$ with $F_i\cap F_j=\emptyset$ for all $i\ne j$.
			\item[(b)] $I^{(k)}=I^k$ for all $k\ge1$.
			\item[(c)] $I$ is packed.
		\end{enumerate}
	\end{Theorem}
	\begin{proof}
		The implication (a) $\Rightarrow$ (b) follows from Remark \ref{Rem} together with the fact that $P^{(k)}=P^k$ for all $k$, if $P$ is a monomial prime ideal. The implication (b) $\Rightarrow$ (c) follows from Proposition \ref{Prop:Pack}.
		
		Let us prove (c) $\Rightarrow$ (a). We proceed by induction on $\alpha(I)=d\ge1$. If $d=1$ there is nothing to prove. Now, let $d>1$. Next, we proceed by induction on $\mu(I)\ge1$. If $\mu(I)=1$, then $I$ is principal and (a) holds. Let $\mu(I)>1$. Without loss of generality we can assume that $I$ is fully-supported. Applying Lemma \ref{Lem:againP} we can write $I=x_1I_1+I_2$ with $I_1,I_2$ matroidal ideals such that $\alpha(I_1)=d-1$, $1\notin\supp(I_1)\cup\supp(I_2)$ and $I_2\subseteq I_1$. Since by assumption $I$ is packed, it follows that putting $x_1=0$ and $x_1=1$ in $I$, the resulting squarefree monomial ideals are also packed. Putting $x_1=0$ in $I$ we obtain that $I_2$ is packed. Whereas, putting $x_1=1$ we obtain that $I_1+I_2=I_1$ is packed. Here we used that $I_2\subseteq I_1$.
		
		If $I_2=(0)$, since $\alpha(I_1)<\alpha(I)$ and $I_1$ is packed, by induction we have that $I_1=P_{G_2}\cdots P_{G_d}$ and so $I=x_1I_1=P_{\{1\}}P_{G_2}\cdots P_{G_d}$ is of the required form.
		
		Now, suppose that $I_2\ne(0)$. Then $\alpha(I_2)=d$ and $\mu(I_2)<\mu(I)$. Recall that $\alpha(I_1)<\alpha(I)$. By induction on $\alpha(I)$ and $\mu(I)$ we have that $I_1=P_{G_1}\cdots P_{G_{d-1}}$ and $I_2=P_{F_1}\cdots P_{F_d}$ with $G_i\cap G_j=\emptyset$ and $F_i\cap F_j=\emptyset$ for all $i\ne j$.
		
		We claim that $\supp(I_2)=[n]\setminus\{1\}$. Otherwise assume that $i\notin\supp(I_2)$ for some $2\le i\le n$. Then, $i\in\supp(I_1)$ because $I$ is fully-supported. Hence, there is a monomial $v\in\mathcal{G}(x_1I_1)\subset\mathcal{G}(I)$ divided by $x_1x_i$. Let $u\in\mathcal{G}(I_2)\subset\mathcal{G}(I)$. We have $\deg_{x_i}(u)<\deg_{x_i}(v)$. By the dual exchange property we can find an integer $j$ with $\deg_{x_j}(u)>\deg_{x_j}(v)$ such that $w=x_i(u/x_j)\in\mathcal{G}(I)$. Since $i\in\supp(w)$ and $i\notin\supp(I_2)$ we have $w\notin I_2$. Hence $w\in x_1I_1$. However this is impossible because $x_1\nmid w$ as $u\in\mathcal{G}(I_2)$. This contradiction proves our claim.
		
		From $I_2\subset I_1$ we have $P_{F_1}\cdots P_{F_d}\subset P_{G_1}\cdots P_{G_{d-1}}\subset P_{G_i}$ for all $i$. Hence, $F_{\ell_i}\subseteq G_i$ for some $\ell_i$. If $i\neq j$, then $\ell_i\neq \ell_j$, otherwise, $F_{\ell_i}\subseteq G_i\cap G_j$, which contradicts to $G_i\cap G_j=\emptyset$. Thus, after a relabeling we have $F_1\subseteq G_1,\ldots,F_{d-1}\subseteq G_{d-1}$. Since $\supp(I_2)=[n]\setminus\{1\}$, we have $F_1\sqcup\dots\sqcup F_d=[n]\setminus\{1\}$. Hence, $G_i\subseteq\bigcup_{j=1}^d(G_i\cap F_j)$. Notice that for $1\leq j\leq d-1$, we have $G_i\cap F_j\subseteq G_i\cap G_j$. So $G_i\cap F_j=\emptyset$ for $j\ne i$ with $1\leq j\leq d-1$. Therefore,
		\begin{equation}\label{incGF}
			G_i\subseteq (G_i\cap F_i)\cup(G_i\cap F_d)=F_i\cup(G_i\cap F_d),
		\end{equation}
		for all $i$. We claim that $F_d$ intersects at most one of the sets $G_1,\ldots,G_{d-1}$. Suppose that this is not the case and let $r\in F_d\cap G_{i_1}$ and $s\in F_d\cap G_{i_2}$ with $i_1\ne i_2$. Then we can find $u\in \mathcal{G}(I_1)$ with $x_rx_s\mid u$. Hence $x_1u\in \mathcal{G}(I)$. Choose $v\in \mathcal{G}(I_2)$. Since $\deg_{x_1}(x_1u)>\deg_{x_1}(v)$, there exists an integer $t$ such that $\deg_{x_t}(x_1u)<\deg_{x_t}(v)$ and $x_t(x_1u)/x_1=x_tu\in \mathcal{G}(I)$. Thus $x_tu\in \mathcal{G}(I_2)$. Since $x_rx_s\mid u$, this means that at least one between $r$ and $s$ belongs to some $F_\ell$ with $\ell\neq d$. Then we get $F_\ell\cap F_d\neq\emptyset$, a contradiction. So the claim is proved and then we may assume that $G_i\cap F_d=\emptyset$ for $2\le i\le d-1$. From this and (\ref{incGF}) we obtain $G_i\subseteq F_i$ for $2\le i\le d-1$. Since we had the other inclusions as well, we get $G_i=F_i$ for $2\le i\le d-1$. Moreover, $F_1\subseteq G_1\subseteq F_1\cup F_d$. Hence
		$$
		I=x_1I_1+I_2=(x_1P_{G_1}+P_{F_1}P_{F_d})P_{F_2}\cdots P_{F_{d-1}}.
		$$
		Putting $x_i=1$ in $I$ for all $i\in F_2\sqcup\cdots\sqcup F_{d-1}$ we obtain the ideal $J=x_1P_{G_1}+P_{F_1}P_{F_d}$ which is packed. Note that $J=I(G)$ is an edge ideal. By \cite[Proposition 4.27]{GRV} it follows that $G$ is bipartite. Since $J=I(P)$, where $P=(x_i: i\in F_1\cup F_d\cup\{1\})$, by~\cite[Corollary 3.2]{HRV} the ideal $J$ is again matroidal. Now, using \cite[Theorem 2.3]{KNQ}  the only matroidal edge ideals are edge ideals of complete multipartite graphs. Since $G$ is bipartite, it follows that $G$ is a complete bipartite graph. Hence $J=P_{A}P_{B}$ with $A\cap B=\emptyset$. Finally $I=P_{A}P_{B}P_{F_2}\cdots P_{F_{d-1}}$ and the sets $A,B,F_2,\dots,F_{d-1}$ are pairwise disjoint because $I$ is squarefree.
	\end{proof}
	
	Now we turn to the non-squarefree case. Non-squarefree polymatroidal ideals $I$ satisfying $I^{(k)}=I^k$ for all $k\ge1$, abound, as we show next. We begin our discussion by stating an elementary lemma which gives a condition ensuring that $I^{(k)}=I^k$ for all $k\ge1$.
	\begin{Lemma}\label{Lem:el}
		Let $I\subset S$ be an ideal with $\m\in\Ass(I)$. Then $I^{(k)}=I^k$ for all $k\ge1$.
	\end{Lemma}
	\begin{proof}
		Since $\m\in\Ass(I)$ and $I^kS_\m\cap S=I^k$ we have $I^{(k)}\subseteq I^k$ for all $k\ge1$. The assertion follows because the opposite inclusion always holds.
	\end{proof}
	
	An immediate consequence of this result is given in the following corollary which allows to construct, starting with any polymatroidal ideal, several polymatroidal ideals whose ordinary and symbolic powers coincide.
	\begin{Proposition}\label{Prop:new-k-k}
		Let $I,J\subset S$ be polymatroidal ideals such that $\m\in\Ass(I)$. Then $(IJ)^{(k)}=(IJ)^k$ for all $k\ge1$.
	\end{Proposition}
	\begin{proof}
		In view of Lemma \ref{Lem:el} it is enough to show that $\m\in\Ass(IJ)$. Since $\m\in\Ass(I)$ we have $\depth(S/I)=0$. By the Auslander-Buchsbaum formula $\pd(S/I)=n$. We have the following isomorphisms
		$$
		\Tor^S_{n-1}(K,I)\cong\Tor^S_n(K,S/I)\cong H_n(x_1,\dots,x_n;S/I)\cong ((I:\m)/I) \,e_1\wedge\cdots\wedge e_n,
		$$
		where $H_n(x_1,\dots,x_n;S/I)$ denotes the $n$th Koszul homology of $S/I$ with respect to the sequence $x_1,\dots,x_n$ and each $e_i$ has degree one. Let $\alpha(I)=d$, $\alpha(J)=d'$. Since $I$ has a $d$-linear resolution, it follows that $\Tor^S_{n-1}(K,I)_j=0$ if and only if $j\ne n+d-1$. Since $\pd(S/I)=n$ it follows that $\Tor^S_{n-1}(K,I)_{n+d-1}\ne0$. Due to the above isomorphisms and since $e_1\wedge\cdots\wedge e_n$ has degree $n$, we deduce that $((I:\m)/I)_{d-1}\ne0$. Hence, we can find a monomial $u\in(I:\m)\setminus I$ of degree $d-1$. Clearly $I:(u)=\m$. Now let $v\in\mathcal{G}(J)$. Then $vu\notin IJ$ because $\deg(uv)=d+d'-1$ but $\alpha(IJ)=d+d'$. Hence the ideal $IJ:(uv)$ is proper and contains $\m$. This shows that $IJ:(uv)=\m$ and so $\m\in\Ass(IJ)$, as desired.
	\end{proof}
	
	In view of Proposition \ref{Prop:new-k-k}, it is hopeless to give a full classification of non-squarefree polymatroidal ideals whose ordinary and symbolic powers coincide.\smallskip
	
	Next, we provide more nice classes of non-squarefree polymatroidal ideals whose ordinary and symbolic powers coincide.\smallskip
	
	Recall that a \textit{transversal} polymatroidal ideal is defined as the product of any arbitrary number of monomial prime ideals.\smallskip
	
	Let $u=x_{i_1}\cdots x_{i_d}\in S$ be a monomial of degree $d$, with $1\le i_1\le\cdots\le i_d\le n$. The \textit{principal Borel ideal} generated by $u$, is defined as the monomial ideal $B(u)\subset S$ such that
	$$
	\mathcal{G}(B(u))=\{x_{j_1}\cdots x_{j_d}:\ j_1\le i_1,\,\dots,\,j_d\le i_d\}.
	$$
	
	It is well-known that $B(u)$ is a polymatroidal ideal. Such an ideal is \textit{strongly stable}. That is, it has the following property: for all monomials $w\in B(u)$ and all $i<j\le n$ with $x_j\mid w$ we have $x_i(w/x_j)\in B(u)$.
	
	\begin{Theorem}\label{Thm:k(k)}
		Consider the following families of polymatroidal ideals.
		\begin{enumerate}
			\item[(a)] Polymatroidal ideals generated in degree $2$ which are not squarefree.
			\item[(b)] Transversal polymatroidal ideals.
			\item[(c)] Principal Borel ideals.
		\end{enumerate}
		Then $I^{(k)}=I^k$ for all $k\geq 1$ for any ideal $I$ belonging to one of the families \textup{(a)-(c)}. In particular, $\reg\,I^{(k)}=\reg\,I^k$ and $I^{(k)}$ has linear quotients, and so it is componentwise linear, for all $k\ge1$.
	\end{Theorem}
	\begin{proof}
		Let $I$ be an ideal belonging to one of the families (a)-(c). Without loss of generality, we may assume that $I$ is fully-supported.
		
		(a) Since $I$ is not squarefree, up to a relabeling we can assume that $x_1^2\in\mathcal{G}(I)$. We claim that $I:(x_1)=\m$. Then the assertion follows Lemma \ref{Lem:el}. Clearly $x_1\in I:(x_1)$. Let $p\in\{2,\dots,n\}$. Since $I$ is fully-supported, we have $x_px_q\in I$ for some $q$. If $q=1$, then $x_p\in I:(x_1)$. Suppose $q\ne 1$. Then $\deg_{x_q}(x_px_q)>\deg_{x_q}(x_1^2)$. Thus, by the exchange property $x_1x_p\in I$ and so $x_p\in I:(x_1)$, as desired.
		
		(b) Let $I=P_1P_2\cdots P_t$ be a transversal polymatroidal ideal, with each $P_i$ a monomial prime ideal. Following \cite[Section 3]{HRV}, we define the \textit{intersection graph} of $I$ as the graph $G_I$ with the vertex set $\{1,\dots,t\}$ such that $\{i,j\}\in E(G_I)$ if and only if $i\ne j$ and $\mathcal{G}(P_i)\cap\mathcal{G}(P_j)\ne\emptyset$. Let $G_1,\dots,G_c$ be the connected components of $G_I$. Then $V(G_i)\cap V(G_j)=\emptyset$ for all $1\le i<j\le c$. It follows that we can write $I=I_1\cdots I_c$, where each $I_i$ is a transversal polymatroidal ideal whose intersection graph $G_{I_i}$ is the graph $G_i$ and such that $\supp(I_i)\cap\supp(I_j)=\emptyset$ for all $1\le i<j\le c$. Regarding each ideal $I_i$ as a monomial ideal in the polynomial ring $S_i=K[x_j:j\in\supp(I_i)]$ with the graded maximal ideal $\m_i$, and since $G_{I_i}=G_i$ is a connected graph, it follows from \cite[Theorem 3.3]{HRV} that $\m_i\in\Ass_{S_i}(I_i)$. Hence, regarding each $I_i$ as an ideal of $S_i$, Lemma \ref{Lem:el} implies that $I_i^{(k)}=I_i^k$ for all $k\ge1$. Since symbolic powers remain unchanged after a polynomial extension of $S_i$ by new variables, it follows that $I_i^{(k)}=I_i^k$ for all $k\ge1$ when regarding the ideals $I_i$ as ideals of $S$. Since the ideals $I_1,\dots,I_c$ have pairwise disjoint supports, using Remark \ref{Rem} it follows that $$I^{(k)}=(I_1\cdots I_c)^{(k)}=I_1^{(k)}\cdots I_c^{(k)}=I_1^{k}\cdots I_c^{k}=I^k$$ for all $k\ge1$, as desired.
		
		(c) Let $I=B(u)$ with $u=x_1^{a_1}\cdots x_n^{a_n}$. Since $I$ is fully-supported we have $a_n>0$. Hence $I:(u/x_n)=\m\in\Ass(I)$ because $I$ is strongly stable. The conclusion follows from Lemma \ref{Lem:el}.
	\end{proof}

	\section{Componentwise linearity of symbolic powers}\label{sec6}
	
	In this final section, we determine several families of polymatroidal ideals which satisfy Conjectures \ref{ConjA} and \ref{ConjB}.
	
	A graph $G$ is called a \textit{complete multipartite graph} if there exists a partition of the vertex set $V(G)=A_1\sqcup\cdots\sqcup A_m$ with $m\ge2$ and $A_i\ne\emptyset$ for $i=1,\dots,m$ such that $E(G)=\{\{x_i,x_j\}:\ x_i\in A_i,\,x_j\in A_j,\,i\ne j\}$. If $m=2$, then $G$ is called a \textit{complete bipartite graph}.
	
	A matroidal ideal generated in degree two is the edge ideal $I=I(G)$ of some graph $G$. It follows from \cite[Theorem 2.3]{KNQ} that $G$ is a complete multipartite graph. Using this fact, in the following proposition we show that Conjectures \ref{ConjA} and \ref{ConjB} hold for such polymatroidal ideals.
	\begin{Proposition}\label{Prop:I(G)-reg}
		Let $I\subset S$ be a matroidal ideal generated in degree two. Then $\reg\,I^{(k)}=\reg\,I^k=2k$ and $I^{(k)}$ is componentwise polymatroidal for all $k\ge1$.
	\end{Proposition}
	\begin{proof}
		We have $I=I(G)$, where $G$ is a complete multipartite graph. Now, let $V(G)=A_1\sqcup\cdots\sqcup A_m$ be the vertex partition of $G$ as described before. It follows that $G^c$ is the disjoint union of the complete graphs on the vertex sets $A_1,\dots,A_m$. Thus $G^c$ is a block graph.  It is shown in the proof of \cite[Theorem 2.3(a)]{FMR} that $(I^{(k)})_{\langle d\rangle}$ has linear quotients with respect to the lex order induced by any perfect elimination order of $G^c$ in the sense of Dirac \cite{Dirac}. Notice furthermore that any order of the variables $x_1,\dots,x_n$ is a perfect elimination order of $G^c$. Hence, combining this with \cite[Theorem 2.4]{BR} we obtain that $I^{(k)}$ is componentwise polymatroidal for all $k$. Finally, by \cite[Corollary 2.4]{FMR} we have $\reg\,I^{(k)}=\reg\,I^k=2k$ for all $k$.
	\end{proof}
	
	As a consequence of this result and Theorem \ref{Thm:k(k)}(a) we have
	\begin{Corollary}
		Let $I\subset S$ be a polymatroidal ideal generated in degree two. Then $\reg\,I^{(k)}=\reg\,I^k$ and $I^k$ has linear quotients for all $k\ge1$.
	\end{Corollary}
	
	We define the \textit{greatest common divisor} of a monomial ideal $I\subset S$ as the monomial $\gcd(I)=\gcd(u:\ u\in\mathcal{G}(I))$. It is clear that any monomial ideal $I\subset S$ can be written as $\gcd(I)J$ for a unique monomial ideal $J\subset S$ with $\gcd(J)=1$. Notice furthermore that $\gcd(I)=1$ if and only if $\height(I)>1$.
	
	\begin{Proposition}\label{Prop:uI}
		Let $I\subset S$ be a monomial ideal and let $u\in S$ be a monomial. Then, $(uI)^{(k)}=(u)^kI^{(k)}$ for all $k\ge1$. In particular, $\reg(uI)^{(k)}=\reg\,I^{(k)}+\deg(u)k$.
	\end{Proposition}
	\begin{proof}
		Firstly, we claim that $\Ass(uI)=\Ass(I)\cup\{(x_i):\ x_i\mid u\}$. To this end, consider the short exact sequence
		\begin{equation}\label{eq:esu}
			0\rightarrow S/((uI):(u))\rightarrow S/(uI)\rightarrow S/(uI,u)\rightarrow0.
		\end{equation}
		Notice that $(uI):(u)=I$ using \cite[Proposition 1.2.2]{HHBook}, and $(uI,u)=(u)$. Then, the exact sequence (\ref{eq:esu}) implies that $\Ass(I)\subseteq\Ass(uI)\subseteq\Ass(I)\cup\Ass((u))$. We have $\Ass((u))=\{(x_i):\ x_i\mid u\}$. Since $uI\subseteq P_{\{i\}}$ for each $i$ with $x_i\mid u$, and $\height(P_{\{i\}})=1$, it follows that $\Ass((u))\subseteq\Ass(uI)$. The claim is proved. 
		
		Now, assume for the moment that $\gcd(I)=1$. Later we discuss the general case. Given a monomial prime ideal $P$ we denote by $u_P$ the monomial of $S$ obtained from $u$ by applying the substitutions $x_i\mapsto 1$ for $x_i\notin P$. Let $u=x_1^{a_1}\cdots x_n^{a_n}$. Using formula (\ref{eq:symbolicP'}) and our claim we have
		\begin{align*}
			(uI)^{(k)}\ &=\ \bigcap_{P\in\Ass(uI)}((uI)(P))^k\\
			&=\ (\bigcap_{P\in\Ass((u))}((uI)(P))^k)\cap(\bigcap_{P\in\Ass(I)}((uI)(P))^k).
		\end{align*}
		
		Let $P\in\Ass((u))$. Then $P=(x_i)$ for some $x_i\mid u$. Since $\gcd(I)=1$, we get that $((uI)(P))^k=(x_i^{ka_i})$. Whereas, if $P\in\Ass(I)$, then $((uI)(P))^k=(u_P^k)(I(P))^k$. Hence,
		\begin{align}
			\nonumber(uI)^{(k)}\ &=\ (\bigcap_{x_i\mid u}(x_i^{ka_i}))\cap(\bigcap_{P\in\Ass(I)}(u_P^k)(I(P))^k)\ =\ (u^k)\cap(\bigcap_{P\in\Ass(I)}(u_P^k)(I(P))^k)\\
			\label{eq:uP1}&=\ \bigcap_{P\in\Ass(I)}(u^k)\cap((u_P^k)(I(P))^k).
		\end{align}
		
		Notice that for any $P\in\Ass(I)$, we can write $u^k=u_P^kv$ with $v=(\prod_{x_i\notin P}x_i^{ka_i})$. Therefore, since $(I(P))^k$ is a monomial ideal whose support is contained in $\supp(P)$ and since $\supp(v)\cap\supp(P)=\emptyset$, it follows that
		\begin{equation}\label{eq:uP2}
			\begin{aligned}
				(u^k)\cap((u_P^k)(I(P))^k) \ &=\ (u_P^kv)\cap((u_P)^k(I(P))^k)\\ &=\ (v)[(u_P)^k\cap((u_P)^k(I(P))^k)]\\
				&=\ (v)(u_P)^k(I(P))^k\ =\ (u)^k(I(P))^k.
			\end{aligned}
		\end{equation}
		Combining (\ref{eq:uP1}) and (\ref{eq:uP2}) we obtain that
		$$
		(uI)^{(k)}=\bigcap_{P\in\Ass(I)}(u^k)(I(P))^k=(u^k)(\bigcap_{P\in\Ass(I)}(I(P))^k)=(u)^kI^{(k)}.
		$$
		
		Suppose now $\gcd(I)\ne 1$. Then $I=vJ$ for some monomial $v\in S$ and a monomial ideal $J\subset S$ with $\gcd(J)=1$. By what we proved before, we have $I^{(k)}=(v)^kJ^{(k)}$ and $(uI)^{(k)}=((uv)J)^{(k)}=(uv)^kJ^{(k)}$. Hence $(uI)^{(k)}=(u)^k(v)^kJ^{(k)}=(u)^kI^{(k)}$.
	\end{proof}
	
	The following elementary result is an immediate consequence of Proposition \ref{Prop:uI}.
	\begin{Corollary}\label{Cor:uI-cl-reg}
		Let $I\subset S$ be a monomial ideal and let $u\in S$ be a monomial. Suppose that $\reg\,I^{(k)}=\reg\,I^k$ and that $I^{(k)}$ has linear quotients for all $k\ge1$. Then $\reg\,(uI)^{(k)}=\reg\,(uI)^k$ and $(uI)^{(k)}$ has linear quotients for all $k\ge1$.
	\end{Corollary}
	
	A polymatroidal ideal $I$ satisfies the \textit{strong exchange property} if for all $u,v\in\mathcal{G}(I)$, all $i,j$ with $\deg_{x_i}(u)>\deg_{x_i}(v)$ and $\deg_{x_j}(u)<\deg_{x_j}(v)$ we have $x_j(u/x_i)\in\mathcal{G}(I)$.
	
	Given an integer $d>0$ and a vector ${\bf a}=(a_1,\dots,a_n)\in\ZZ_{\ge0}^n$, the ideal of \textit{Veronese type} $(n,d,{\bf a})$ is defined as the ideal $I_{n,d,{\bf a}}$ such that
	$$
	\mathcal{G}(I_{n,d,{\bf a}})\ =\ \big\{x_1^{b_1}\cdots x_n^{b_n}\ :\ \sum_{i=1}^{n}b_i=d\ \textup{and}\ b_i\le a_i\ \textup{for}\ i=1,\dots,n\big\}.
	$$
	
	It is known that a polymatroidal ideal $I$ satisfies the strong exchange property if and only if $I=(u)I_{n,d,{\bf a}}$ for some monomial $u\in S$, see \cite[Theorem 1.1]{HHV2005}.
	
	\begin{Proposition}\label{Prop:three}
		Let $I\subset S$ be a polymatroidal ideal in at most three variables. Then, $\reg\,I^{(k)}=\reg\,I^k$ and $I^{(k)}$ has linear quotients, for all $k\ge1$.
	\end{Proposition}
	\begin{proof}
		If $|\supp(I)|=1$, then $I$ is a principal ideal and there is nothing to prove. Let $|\supp(I)|=2$. Again, we may suppose that $I$ is not a principal ideal. It follows from \cite[Proposition 5.1(c)]{FS2} that $\m=(x_1,x_2)\in\Ass(I)$. Then the assertion follows from Lemma \ref{Lem:el}. Now suppose that $\supp(I)=[3]$. By \cite[Proposition 2.7]{BH2013}, $I$ satisfies the strong exchange property. Hence, $I=uI_{3,d,{\bf a}}$ with $u\in S$ a monomial and ${\bf a}=(a_1,a_2,a_3)\in\ZZ_{\ge0}^d$ such that $a_1+a_2+a_3\ge d$. Using Corollary \ref{Cor:uI-cl-reg} we may assume that $u=1$. Hence, we can assume that $I=I_{3,d,{\bf a}}$. Since $I$ is fully-supported we have $a_i>0$ for $i=1,2,3$. If $a_1+a_2+a_3=d$, then $I$ is principal and there is nothing to prove. Suppose now $a_1+a_2+a_3\ge d+1$. By \cite[Corollary 4.5]{HRV} we have that $\m=(x_1,x_2,x_3)\in\Ass(I)$ if and only if $a_1+a_2+a_3\ge d+2$. Hence, if $a_1+a_2+a_3\ge d+2$, then Lemma \ref{Lem:el} implies that $I^{(k)}=I^k$ for all $k\ge1$, and there is nothing to prove. Now, suppose that $a_1+a_2+a_3=d+1$. Then,
		$$
		I=(x_1^{a_1}x_2^{a_2}x_3^{a_3-1},x_1^{a_1}x_2^{a_2-1}x_3^{a_3},x_1^{a_1-1}x_2^{a_2}x_3^{a_3})=x_1^{a_1-1}x_2^{a_2-1}x_3^{a_3-1}(x_1x_2,x_1x_3,x_2x_3).
		$$
		Notice that $(x_1x_2,x_1x_3,x_2x_3)$ is the edge ideal of a complete graph on three vertices. Hence, the assertion follows by combining Proposition \ref{Prop:I(G)-reg} with Corollary \ref{Cor:uI-cl-reg}.
	\end{proof}
	
	Next, we consider matroidal ideals generated in small degrees. We are able to establish Conjectures \ref{ConjA} and \ref{ConjB} for all matroidal ideals in at most four variables.

	\begin{Proposition}\label{Prop:four}
		Let $I\subset S$ be a matroidal ideal in at most four variables. Then, $\reg\,I^{(k)}=\reg\,I^k$ and $I^{(k)}$ has linear quotients, for all $k\ge1$.
	\end{Proposition}
	\begin{proof}
		Let $I\subset S$ be a matroidal ideal in a most four variables. Then $\alpha(I)\le4$. Notice that $\alpha(I)=1$ if and only if $I$ is a monomial prime ideal and $\alpha(I)=|\supp(I)|$ if and only if $I=(\prod_{i\in\supp(I)}x_i)$. In these cases the assertion holds. If $\alpha(I)=2$, then $I$ is a matroidal edge ideal and the assertion follows from Proposition \ref{Prop:I(G)-reg}. Therefore if $|\supp(I)|\le3$ there is nothing to prove.
		
		Suppose now $\supp(I)=[4]$. We only have to consider the case $\alpha(I)=3$. Using Lemma \ref{Lem:againP} we can write $I=x_4I_1+I_2$, with $4\notin\supp(I_2)$, $I_2\subset I_1$ and $I_1,I_2$ are polymatroidal. If $I_2=(0)$, then $I$ satisfies the assertion by using Corollary \ref{Cor:uI-cl-reg} because $I_1$ satisfies the assertion by our discussion. Suppose $I_2\ne(0)$. Since $\supp(I_2)\subseteq[3]$ and $\alpha(I_2)=3$ we have $I_2=(x_1x_2x_3)$. Notice that $\alpha(I_1)=2$. Therefore, by using \cite[Theorem 2.3]{KNQ}, $I_1$ is the edge ideal of a complete multipartite graph on two or on three vertices. Hence, up to a relabeling, we have the next three possible cases: (i) $I_1=(x_1x_2)$, (ii) $I_1=(x_1x_2,x_1x_3)$ or (iii) $I_1=(x_1x_2,x_1x_3,x_2x_3)$. Therefore, we have the following cases:
		\begin{enumerate}
			\item[(i)] $I=x_4(x_1x_2)+(x_1x_2x_3)=(x_1)(x_2)(x_3,x_4)$ is a transversal polymatroidal ideal and it satisfies the assertion by Theorem \ref{Thm:k(k)}(c).
			\item[(ii)] $I=x_4(x_1x_2,x_1x_3)+(x_1x_2x_3)=I_{\{1\},1}I_{\{2,3,4\},2}$ is a matching-matroidal ideal of Veronese type and it satisfies the assertion by Theorem \ref{Thm:mat-mat}.
			\item[(iii)] $I=x_4(x_1x_2,x_1x_3,x_2x_3)+(x_1x_2x_3)=I_{4,3}$ is a squarefree Veronese ideal and it satisfies the assertion by Corollary \ref{Cor:Veronese}.
		\end{enumerate}
		The proof is complete.
	\end{proof}
	
	As the final conclusion of the paper, in the next result we summarize the families of polymatroidal ideals for which we know that Conjectures \ref{ConjA} and \ref{ConjB} hold.
	\begin{Theorem}
		Conjectures \ref{ConjA} and \ref{ConjB} hold for the following families of polymatroidal ideals.
		\begin{enumerate}
			\item[(a)] Polymatroidal ideals generated in degree two.
			\item[(b)] Polymatroidal ideals in at most three variables.
			\item[(c)] Matroidal ideals in at most four variables. 
			\item[(d)] Transversal polymatroidal ideals.
			\item[(e)] Squarefree Veronese ideals.
			\item[(f)] Matching-matroidal ideals of Veronese type.
			\item[(g)] Principal Borel ideals.
		\end{enumerate}
		Furthermore, Conjectures \ref{ConjA} and \ref{ConjB} hold for any product of polymatroidal ideals, having pairwise disjoint supports, belonging to the families \textup{(a)-(g)}.  
	\end{Theorem}\smallskip
	
	\noindent\textbf{Acknowledgment.}
	A. Ficarra was partly supported by INDAM (Istituto Nazionale di Alta Matematica), and also by the Grant JDC2023-051705-I funded by
	MICIU/AEI/10.13039/501100011033 and by the FSE+. S. Moradi is supported by the Alexander von Humboldt Foundation.

\end{document}